\documentclass[aos, preprint]{imsart}

\RequirePackage{amsthm,amsmath,amsfonts,amssymb}
\RequirePackage[authoryear]{natbib}
\RequirePackage[colorlinks,citecolor=blue,urlcolor=blue]{hyperref}
\RequirePackage{graphicx}

\startlocaldefs
\usepackage[utf8]{inputenc}
\usepackage{hyperref}
\usepackage{microtype}

\usepackage{amsfonts}
\usepackage{amssymb}
\usepackage{bm}
\usepackage{enumerate}
\usepackage[dvipsnames]{xcolor}

\usepackage{calrsfs}
\usepackage{dsfont}
\newcommand{\1}{\mathds{1}}

\newtheorem{theorem}{Theorem}
\newtheorem{condition}{Condition}
\newtheorem{lemma}{Lemma}
\newtheorem{corollary}{Corollary}

\theoremstyle{remark}
\newtheorem{remark}{Remark}
\newtheorem{example}{Example}

\renewcommand{\ge}{\geqslant}
\renewcommand{\geq}{\ge}
\renewcommand{\le}{\leqslant}
\renewcommand{\leq}{\le}

\newcommand{\vc}[1]{\bm{#1}}
\newcommand{\eps}{\varepsilon}
\newcommand{\reals}{\mathbb{R}}
\newcommand{\diff}{\mathrm{d}}
\newcommand{\point}{\,\cdot\,}
\newcommand{\ceil}[1]{\lceil{#1}\rceil}

\newcommand{\prob}{\operatorname{\mathbb{P}}}
\newcommand{\expec}{\mathbb{E}}
\newcommand{\var}{\operatorname{var}}
\newcommand{\cov}{\operatorname{cov}}

\newcommand{\dto}{\rightsquigarrow}

\newcommand{\abs}[1]{\left\lvert{#1}\right\rvert}
\newcommand{\norm}[1]{\left\lVert{#1}\right\rVert}

\newcommand{\je}[1]{\bgroup\color{Bittersweet}\small\sffamily [JE: {#1}]\egroup}
\newcommand{\js}[1]{\bgroup\color{MidnightBlue}\small\sffamily [JS: {#1}]\egroup}
\newcommand{\modiv}[1]{\bgroup\color{black}#1\egroup}
\endlocaldefs

\begin{document}

\begin{frontmatter}
\title{Empirical tail copulas for functional data}
\runtitle{Empirical tail copulas for functional data}

\begin{aug}
\author[A]{\fnms{John H.\!\! J.} \snm{Einmahl}\ead[label=e1]{j.h.j.einmahl@uvt.nl}}
\and
\author[B]{\fnms{Johan} \snm{Segers}\ead[label=e2]{johan.segers@uclouvain.be}}

\address[A]{Department of Econometrics and OR and CentER, Tilburg University, \printead{e1}}
\address[B]{LIDAM/ISBA, UCLouvain, \printead{e2}}

\runauthor{J.H.J. Einmahl and J. Segers}
\end{aug}

\begin{abstract}
For multivariate distributions in the domain of attraction of a max-stable distribution, the tail copula and the stable tail dependence function are equivalent ways to capture the dependence in the upper tail. The empirical versions of these functions are rank-based estimators whose inflated estimation errors are known to converge weakly to a Gaussian process that is similar in structure to the weak limit of the empirical copula process. We extend this  multivariate result to continuous functional data by establishing the  asymptotic normality  of the estimators of the tail copula, uniformly over all finite subsets of at most $D$ points ($D$ fixed).
An application for testing tail copula stationarity is presented. The main tool for deriving the result is the uniform asymptotic normality of all the $D$-variate tail empirical processes. The proof of the main result is non-standard.
\end{abstract}


\end{frontmatter}

\section{Introduction}

Consider the statistical theory of  extreme values for functional data taking values in $C([0,1])$, the space of continuous functions on the interval $[0,1]$. Max-stability in $C([0,1])$ was characterized in \citet{gine:1990} and  the corresponding domain of attraction conditions have been studied in \citet{dhl1}. Consistency of  extreme value index estimator functions and of estimators of the exponent measure was established in \citet{dhl2} and asymptotic normality of
extreme value index estimator functions was obtained in \citet{einmahl+l:2006}.

In this paper we extend the  setup of the latter two papers by replacing the interval $[0,1]$ (with absolute value metric)  with  a general, compact metric space $(T,\rho)$. Interesting choices for $T$ can be the unit square, the unit cube, or the unit hypercube in any finite dimension as well as a general, compact set in dimension two, three, or higher. Statistical inference is based on independent and identically distributed data that are random functions with values in $C(T)$. Now consider multivariate extreme value theory by restricting the data to a finite subset of $d$ points in $T$.  We assume the existence of the $d$-variate tail copula $R$ and the stable tail dependence function $l$, each of which determines the tail dependence structure of the multivariate data, see, e.g., \citet{dh:1998}, \citet{schmidt+s:2006}, \citet{peng+q:2008}, \citet{nikoloulopoulos+j+l:2009} and \citet{bucher+d:2013}. The asymptotic normality, uniform in the argument, of the usual rank-based estimator of $l$ can be found in, e.g., \citet{einmahl+k+s:2012}.

It is the main purpose of this paper to extend the latter multivariate result to random functions in $C(T)$ by considering the  asymptotic normality  of the rank-based estimators of the tail copula $R$, uniformly over all finite subsets of $T$ of at most $D$ points ($D$ fixed). To the best of our knowledge this is the first paper that establishes  asymptotic normality of estimators of the tail dependence structure for functional data in this general, nonparametric setting. Note that as a special case we obtain the uniform asymptotic normality of all estimated upper tail dependence coefficients (corresponding to two points in $T$). The main tool for the main result  is the uniform asymptotic normality of the multivariate tail empirical processes, defined for all $D$-tuples in $T$. This  result is of independent interest and generalizes a  similar result  for $D=1$ (and $T=[0,1]$) in \citet{einmahl+l:2006}.

As in the multivariate case, the limiting process for the tail copula estimators requires the existence and continuity of the partial derivatives of the tail copula which yields an unexpected, challenging problem for functional data. The problem arises since  points in $T$ that are close have, because of strong dependence due to continuity, a bivariate tail copula that is close to the non-differentiable ``comonotone" tail copula, being the minimum of the two coordinates. This phenomenon requires a novel, non-standard proof-technique  for  the main result.

The fact that our results are uniform over $T$ sets them apart from those that fix a finite vector of locations in $T$. Thanks to the uniformity, it becomes possible to treat functionals that depend on all $t \in T$ simultaneously, such as suprema and integrals. By way of motivation, some possible applications of our results are sketched below.
\begin{itemize}
	\item
	In order to detect non-stationarity of the extreme-value attractor of a random field over $T$, a subset of Euclidean space, \citet{dombry:2017} suggests to compute an integral $\hat I(t_0)$ of estimated tail dependence coefficients at pairs $(t_0,t) \in T^2$, for fixed $t_0$ and integrating over $t$ in a neighbourhood of $t_0$. Stationarity then yields that the function $t_0 \mapsto \hat I(t_0)$ is about constant. Our main result yields the asymptotic behavior of $\hat I(t_0)$.  We elaborate on this idea in Section \ref{sec:testStat} where we  present as a test statistic the range of a  discrete version of these integrals,   assuming that we observe the functional data only on a fine, but finite grid on $T=[0,1]$.    A similar integral, but then based on the notion of concurrence probabilities, is used in \citet{dombry+r+s:2018} to illustrate the spatial structure of extremes.
	\item
	\citet{koch:2017} introduces spatial risk measures, some of which can be expressed as double integrals of pairwise extremal coefficients over pairs of locations. Our weak convergence theory yields the asymptotic normality of nonparametric estimates of these risk measures.
	\item
	Matching nonparametric tail copula estimators to a parametric form yields minimum-distance estimators of parameters of tail dependence models. For multivariate data, such estimators are proposed and analyzed in \citet{einmahl+k+s:2012} and \citet{einmahl+k+s:2017}.  Our uniform central limit theorem implies the asymptotic normality  of such estimators in case of functional data.
	\item
	The goodness-of-fit of a parametric family of spatial tail dependence models may be assessed via  a  supremum over   $(t_1, \ldots, t_D) \in T^D$ of  Kolmogorov--Smirnov type test statistics based on  the difference between a nonparametric and a parametric estimator  of the $D$-variate tail copula.  Our results not only yield the asymptotic distribution of the test statistic under the null hypothesis, they even cover the setting of infill asymptotics, where the process is observed only on a finite, possibly random set $T_n \subset T$ of locations, with the property that $T_n$ becomes dense in $T$ as $n$ grows (formally, for every open subset $G \subset T$, the probability that $T_n$ intersects $G$ increases to $1$ as $n \to \infty$).
\end{itemize}

The paper is organized as follows. The joint asymptotic behavior of all the  multivariate tail empirical processes is studied in Section~\ref{sec:tailempproc}. The main result on the joint asymptotic behavior of all the empirical tail copulas is presented in Section~\ref{sec:emptailcop}.  The application to testing stationarity of the tail copula is developed in Section~\ref{sec:testStat}. The proofs of the results in Sections~\ref{sec:tailempproc} and~\ref{sec:emptailcop} are deferred to Sections~\ref{sec:proofs:tail} and~\ref{sec:proofs:emp}, respectively. Some auxiliary results are collected in Appendices~\ref{app:aux} and~\ref{app:testStat}.


\section{Tail empirical processes}
\label{sec:tailempproc}

Let $(T, \rho)$ be a compact metric space; in \citet{einmahl+l:2006}, the space $T$ is $[0, 1]$ with the absolute value metric. Let $C(T)$ be the space of continuous functions $f : T \to \reals$.
Actually the metric $\rho$ on $T$ does not play much of a role. It is merely introduced to have a complete, separable subset of $\ell^\infty(T)$, the space of bounded functions $f : T \to \reals$, facilitating the study of stochastic processes on $T$. Let $C(T)$ be equipped with the supremum norm and distance and the associated Borel $\sigma$-field.

Let $\xi_1, \xi_2, \ldots$ be i.i.d.\ random elements in $C(T)$. For $(t, x) \in T \times \reals$, define $F_t(x) = \prob\{ \xi_i(t) \le x \}$. Since the trajectories of $\xi_i$ are continuous, the map $t \mapsto F_t$ is continuous with respect to the topology of weak convergence, that is, if $t_n \to t$ in $T$ then $F_{t_n} \to F_t$ weakly as $n \to \infty$. We will consider the (inverse) probability integral transform at each coordinate $t \in T$, and this requires the following condition. 

\begin{condition}
	\label{cond:cont}
	The distribution functions $F_t$, $t \in T$, are continuous.
\end{condition}

Under Condition~\ref{cond:cont}, the random variables $U_i(t) = 1 - F_t(\xi_i(t))$ for $i = 1, 2, \ldots$ and $t \in T$ are uniformly distributed on $(0, 1)$. Moreover, the map $t \mapsto F_t$ is now continuous with respect to the supremum distance: if $t_n \to t$ in $T$ then $\sup_{x \in \reals} \lvert F_{t_n}(x) - F_t(x) \rvert \to 0$ as $n \to \infty$. The stochastic processes $U_1, U_2, \ldots$ are i.i.d.\ random elements of $C(T)$ too. Indeed, we have $U_i = 1-F(\xi_i)$ where $F : C(T) \to C(T)$ is defined by $(F(x))(t) = F_t(x(t))$ for $x \in C(T)$ and $t \in T$. The map $F$ is continuous and thus Borel measurable.


For $d = 1, 2, \ldots$ and $\vc{t} = (t_1, \ldots, t_d) \in T^d$, define the $d$-variate empirical distribution function $S_{n,\vc{t}}$ by
\begin{equation}
\label{eq:Snt}
	S_{n,\vc{t}}(\vc{u})
	= \frac{1}{n} \sum_{i=1}^n \1\{ \forall j = 1, \ldots, d : U_i(t_j) < u_j \}
\end{equation}
for $\vc{u} = (u_1, \ldots, u_d) \in [0, 1]^d$. Note that $S_{n,\vc{t}}$ involves the marginal distribution functions $F_{t_j}$, which are unknown. In Section~\ref{sec:emptailcop} we will replace these functions by their empirical counterparts. But to study the resulting statistic, we first need to consider the case where the functions $F_{t_j}$ are known.

Fix $D = 1, 2, \ldots$ and $M > 0$ and put $\mathcal{V} = \bigcup_{d = 1}^D \mathcal{V}_d$ where $\mathcal{V}_d = T^d \times [0, M]^d$ for $d = 1, \ldots, D$. Let $k=k_n \in (0, n]$ and consider
\begin{align}
	\label{eq:Gntx}
	G_{n,\vc{t}}(\vc{x})
	&=
	\tfrac{n}{k} S_{n,\vc{t}}(\tfrac{k}{n} \vc{x})
	=
	\frac{1}{k} \sum_{i=1}^n \1 \left\{
		\forall j = 1, \ldots, d: U_i(t_j) < \tfrac{k}{n} x_j
	\right\}, \\
	\label{eq:Entx}
	E_{n,\vc{t}}(\vc{x})
	&=
	\expec[G_{n,\vc{t}}(\vc{x})]
	=	
	\tfrac{n}{k} \prob\{\forall j = 1, \ldots, d: U_i(t_j) < \tfrac{k}{n} x_j\}
\end{align}
for $\vc{v} = (\vc{t}, \vc{x}) \in \mathcal{V}$. The tail empirical process $W_{n}$ indexed by $\mathcal{V}$ is defined by
\[
	W_{n}(\vc{v})
	=
	\sqrt{k} \left\{
		G_{n,\vc{t}}(\vc{x}) - E_{n,\vc{t}}(\vc{x})
	\right\},
	\qquad \vc{v} = (\vc{t},\vc{x}) \in \mathcal{V}.
\]

\begin{condition}
	\label{cond:kn}
	The sequence $k = k_n \in (0, n]$ is such that $k \to \infty$ and $k/n\to 0$ as $n \to \infty$.
\end{condition}

We will investigate weak convergence of $W_n$ in $\ell^\infty(\mathcal{V})$. To control its finite-dimensional distributions, we will assume the following tail dependence condition. Note that we need to specify it up to dimension $2D$ in order to control the relevant covariances.

\begin{condition}
	\label{cond:tailCopula}
	For all $m =  2, \ldots, 2D$ and all $(\vc{t}, \vc{x}) \in T^m \times [0, \infty)^m $, the following limit exists:
	\[
		R_{\vc{t}}(\vc{x})
		=
		\lim_{s \downarrow 0}
		s^{-1} \, \prob\{ U_1(t_1) \le sx_1, \ldots, U_1(t_m) \le sx_m \}.
	\]
\end{condition}

Observe that for $m = 1$, we have $R_t(x) = x$ for all $(t, x) \in T \times [0, \infty)$. For $D=1,2, \ldots$, the function $R_{\vc{t}} : ([0, \infty]^D \setminus \{\vc{\infty}\}) \to [0, \infty)$ is called the (lower and upper, respectively) tail copula of the random vectors $(U_1(t_1), \ldots, U_1(t_D))$ and $(\xi_1(t_1), \ldots, \xi_1(t_D))$.
It is the main subject of this paper.
The tail copula, together with the marginal distributions, determines a multivariate extreme-value distribution $G$, say. In other words, it determines the dependence structure of $G$ and hence the tail dependence structure of a probability distribution that is in the max-domain of  attraction of $G$. More explanations on and various properties and representations of the tail copula and the closely related stable tail dependence function defined in Remark~\ref{rem:stdf} below are given in, e.g., \citet{dh:1998}, \citet{nikoloulopoulos+j+l:2009}, \citet{ressel:2013}, \citet{mercadier+r:2019}, and \citet{falk:2019}.

The process $W_n$ is centered. Let us calculate its covariance function. Write $\vc{U}_1(\vc{t}) = (U_1(t_1), \ldots, U_1(t_m))$ for $\vc{t} \in T^m$ and let inequalities between vectors be meant coordinatewise.
For $\vc{v}_j = (\vc{t}_j, \vc{x}_j) \in \mathcal{V}$, with $j = 1, 2$, we have
\begin{align}
\nonumber
	\lefteqn{
		\expec[ W_n(\vc{v}_1) W_n(\vc{v}_2) ]
	} \\
\nonumber
	&=
	\tfrac{n}{k} \left[
		\prob\left\{
			\vc{U}_1(\vc{t}_1, \vc{t}_2)
			\le \tfrac{k}{n} (\vc{x}_1, \vc{x}_2)
		\right\}
		-
		\prob\left\{ \vc{U}_1(\vc{t}_1) \le \tfrac{k}{n} \vc{x}_1 \right\} \,
		\prob\left\{ \vc{U}_1(\vc{t}_2)\le \tfrac{k}{n} \vc{x}_2 \right\}
	\right] \\
\label{eq:covn}
	&\to
	R_{\vc{t}_1, \vc{t}_2}(\vc{x}_1, \vc{x}_2), \qquad n \to \infty,
\end{align}
in view of Conditions~\ref{cond:kn} and~\ref{cond:tailCopula}. Note that the vectors $\vc{v}_1$ and $\vc{v}_2$ may have different lengths, say $2d_1$ and $2d_2$, with $d_1,d_2\in \{1, \ldots, D\}$, and that the limit in Eq.~\eqref{eq:covn} involves the tail copula of Condition~\ref{cond:tailCopula} in dimension $m = d_1 + d_2$.

Let $W = (W(\vc{v}))_{\vc{v} \in \mathcal{V}}$ be a zero-mean Gaussian process defined on $\mathcal{V}$ with covariance function
\begin{equation}
\label{eq:cov}
	\expec[ W(\vc{v}_1) W(\vc{v}_2) ]
	=
	R_{\vc{t}_1, \vc{t}_2}(\vc{x}_1, \vc{x}_2),
	\qquad
	\vc{v}_j = (\vc{t}_j, \vc{x}_j) \in \mathcal{V}, \quad j = 1, 2.
\end{equation}
Because of \eqref{eq:covn}, the function $(\vc{v}_1, \vc{v}_2) \mapsto R_{\vc{t}_1, \vc{t}_2}(\vc{x}_1, \vc{x}_2)$ is the pointwise limit as $n \to \infty$ of a sequence of covariance functions on $\mathcal{V} \times \mathcal{V}$ and therefore positive semidefinite, so that a Gaussian process $W$ with covariance function in Eq.~\eqref{eq:cov} indeed exists. Further, note that $R_{\vc{t}, \vc{t}}(\vc{x}, \vc{x}) = R_{\vc{t}}(\vc{x}) = \var \{ W(\vc{v}) \}$. The standard deviation semimetric on $\mathcal{V}$ is
\begin{align}
\label{eq:rho2}
	\rho_{2}(\vc{v}_1, \vc{v}_2)
	&=
	\bigl(\expec[\{W(\vc{v}_1) - W(\vc{v}_2)\}^2]\bigr)^{1/2} \\
\nonumber
	&=
	\bigl( R_{\vc{t}_1}(\vc{x}_1) - 2 R_{\vc{t}_1, \vc{t}_2}(\vc{x}_1, \vc{x}_2) + R_{\vc{t}_2}(\vc{x}_2) \bigr)^{1/2}.
\end{align}

To show the asymptotic tightness of the processes $W_n$,
%
%
we need to control the local increments of $U_i$. The following condition is inspired by \citet[Eq.~(16)]{einmahl+l:2006} but is weaker, see Remark~\ref{rem:Esdelta} below.

\begin{condition}
	\label{cond:cover:T}
	There exist positive scalars $c_1$, $c_2$, $u_0$, and $\eps_0$, and, for all $\eps \in (0, \eps_0 ]$, a covering of $T = \bigcup_{j=1}^{N_{T,\eps}} T_{\eps,j}$ by sets $T_{\eps,j}$ such that for all $(u, \eps, j) \in (0, u_0] \times (0, \eps_0] \times \{1, \ldots, N_{T,\eps}\}$, we have
	\begin{equation}
	\label{eq:cond:cover:T}
		\prob \left[
			\sup_{t \in T_{\eps,j}} U_1(t) > u \exp(c_1 \eps^2)
			\, \left\vert \,
			\inf_{t \in T_{\eps,j}} U_1(t) \le u
			\right.
		\right]
		\le
		c_2 \eps^2
	\end{equation}
	and such that $\int_0^1 \sqrt{\log N_{T,\eps}} \, \diff \eps < \infty$.
\end{condition}

In \eqref{eq:cond:cover:T}, we could also replace $\eps^2$ by $\eps$ and require $\int_0^1 \sqrt{ \log N_{T,\eps^2} } \, \diff \eps < \infty$.

\begin{theorem}
	\label{thm:main}
	Let $\xi_1, \xi_2, \ldots$ be i.i.d.\ random elements in $C(T)$, with $(T, \rho)$ a compact metric space. Suppose Conditions~\ref{cond:cont} to~\ref{cond:cover:T} hold and let $D = 1, 2, \ldots$ and $M > 0$. Then the semimetric space $(\mathcal{V}, \rho_2)$ is totally bounded and there exists a version of $W$ of which almost all sample paths are uniformly $\rho_2$-continuous. Moreover, we have the weak convergence
	\[
		W_n \rightsquigarrow W
		\text{ as } n \to \infty
		\text{ in }  \ell^\infty(\mathcal{V}).
	\]
\end{theorem}

Note that, moreover, the processes $W_n$ are asymptotically uniformly equicontinuous in probability with respect to $\rho_2$, see Addendum~1.5.8 and Example~1.5.10 in \cite{vandervaart+w:1996}.


So far, the original metric $\rho$ on $T$ has played hardly a role. The covering sets $T_{\eps,j}$ in Condition~\ref{cond:cover:T} do not involve $\rho$ and uniform continuity of the sample paths of the limit process $W$ in Theorem~\ref{thm:main} is with respect to the standard deviation metric. We can make the link more explicit by requiring that the sets $T_{\eps,j}$ are balls in the original metric $\rho$.

\begin{condition}
	\label{cond:cover:balls}
	There exist positive scalars $c_1$, $c_2$, $u_0$ and $\eps_0$, and, for all $\eps \in (0, \eps_0]$, a scalar $\delta(\eps) > 0$, such that for all $(u, \eps, s) \in (0, u_0] \times (0, \eps_0] \times T$ we have
	\[
		\prob \left[
			\sup_{t \in T : \rho(s, t) \le \delta(\eps)} U_1(t)
				> u \exp(c_1 \eps^2)
			\, \left\vert \,
			\inf_{t \in T : \rho(s, t) \le \delta(\eps)} U_1(t)
				\le u
			\right.
		\right]
		\le c_2 \eps^2.
	\]
	Moreover, if $N_{T,\delta}$ denotes the minimal number of closed balls with radius $\delta$ needed to cover $T$, then we have $\int_0^1 \sqrt{\log N_{T,\delta(\eps)}} \, \diff \eps < \infty$.
\end{condition}

 Condition~\ref{cond:cover:balls} implies Condition~\ref{cond:cover:T}; the   sets $T_{\eps,j}$ in  Condition~\ref{cond:cover:T} can be chosen to be the $N_{T,\delta(\eps)}$ balls of radius $\delta(\eps)$ in
Condition~\ref{cond:cover:balls}. It is a rather weak condition. For natural examples it is amply satisfied,  see Examples \ref{exs} and \ref{exp} below.

\begin{corollary}
	\label{cor:main}
	The conclusions of Theorem~\ref{thm:main} continue to hold if Condition~\ref{cond:cover:T} is replaced by Condition~\ref{cond:cover:balls}. In that case, there exists for every $\eta > 0$ a scalar $\delta > 0$ such that for all $(s, t) \in T^2$ with $\rho(s, t) < \delta$, we have $1 - R_{s,t}(1, 1) < \eta$. As a consequence, the function $(\vc{t}, \vc{x}) \mapsto R_{\vc{t}}(\vc{x})$ on $T^d \times [0, \infty)^d$ equipped with the product topology induced by $\rho$ and the absolute value metric is jointly continuous.
\end{corollary}

\begin{remark}[Increments]
	\label{rem:Esdelta}
	We show that for $T = [0, 1]$ equipped with the absolute value metric, Condition~\ref{cond:cover:T} is more general than Eq.~(16) in \citet{einmahl+l:2006} (with $\beta = 0$).
Suppose that there exist constants $K, c > 0$ and $q > 1$ such that for all sufficiently small $u, \delta > 0$, we have, for all $s \in [0, 1]$,
	\begin{multline}
	\label{eq:einmahl+l:2006}
		\prob \left[
			\sup_{t \in [s, s+\delta]} \frac{\abs{U_1(t) - U_1(s)}}{U_1(t)}
			> K \{\log(\delta^{-1})\}^{-q}
			\, \left\vert \,
			\inf_{t \in [s, s+\delta]} U_1(t) \le u			
			\right.
		\right] \\
		\le
		c \{\log(\delta^{-1})\}^{-q}.
	\end{multline}
 Eq.~(16) in \citet{einmahl+l:2006} implies Eq.~\eqref{eq:einmahl+l:2006} with $q = 2$.
We show that Eq.~\eqref{eq:einmahl+l:2006} implies Condition~\ref{cond:cover:T}.
 For $\eps > 0$, put $\delta = \delta(\eps) = \exp(-\eps^{-2/q})$ and consider the covering $T_{\eps,j} = [(j-1)\delta, j\delta] \cap [0, 1]$ for $j = 1, \ldots, \lceil{1/\delta}\rceil$. The number of covering sets is $N_{T,\eps} = \lceil{1/\delta}\rceil = \lceil{\exp(\eps^{-2/q})}\rceil$, so that $\sqrt{\log N_{T,\eps}} = O(\eps^{-1/q})$ as $\eps \downarrow 0$, which is integrable near zero since $q > 1$. By definition, we have $\{\log(1/\delta)\}^{-q} = \eps^2$.
	  If there exists $\eta > 0$ such that $\sup_{t \in [s, s+\delta]} \abs{U_1(t) - U_1(s)} / U_1(t) \le \eta$, then $ U_1(s)/(1+\eta) \le U_1(t) \le  U_1(s)/(1-\eta)$ for all $t \in [s, s+\delta]$, and thus $\sup_{t \in [s, s+\delta]} U_1(t) \le \frac{1+\eta}{1-\eta} \inf_{t \in [s, s+\delta]} U_1(t)$. It follows that
	\begin{multline*}
		\prob \left[
			\sup_{t \in [s, s+\delta]} U_1(t) > \tfrac{1+\eta}{1-\eta} u
			\, \left\vert \,
			\inf_{t \in [s, s+\delta]} U_1(t) \le u
			\right.
		\right] \\
		\le
		\prob \left[
			\sup_{t \in [s, s+\delta]}
				\frac{\abs{U_1(t) - U_1(s)}}{U_1(t)} > \eta
			\, \left\vert \,
			\inf_{t \in [s, s+\delta]} U_1(t) \le u
			\right.
		\right].
	\end{multline*}
	Setting $\eta = K \eps^2$, we have $\exp(c_1 \eps^2) \ge \frac{1+\eta}{1-\eta}$ for $c_1 = 3K$ and all $\eps > 0$ small enough.  Condition~\ref{cond:cover:T} follows.
(Actually the stronger  Condition~\ref{cond:cover:balls} also readily follows, essentially by relabeling $[s,s+\delta]$ by $[s-\delta, s+\delta]$.) 
\end{remark}
\vspace{-.6 cm}
\begin{remark}[Function spaces]
	\label{rem:setup}
	Rather than with the space $C(T)$ of continuous functions on a compact metric space $(T, \rho)$, we could work with the space of uniformly continuous functions on a totally bounded semimetric space. However, by \citet[Example~18.7]{vandervaart:1998}, this would not really be more general, since we could always consider the completion of the space, which is compact, and extend the functions appropriately. To pass from a semimetric (allowing distinct points to be at distance zero) to a genuine metric can be achieved by considering the appropriate quotient space.
	
	We could even work with stochastic processes on $T$ whose common distribution is a tight Borel measure on $\ell^\infty(T)$, the space of bounded functions $f : T \to \reals$. However, by \citet[Addendum~1.5.8]{vandervaart+w:1996}, we could then find a semimetric $\rho$ on $T$ such that $(T, \rho)$ is totally bounded and such that the processes have uniformly $\rho$-continuous trajectories. By the previous paragraph, this could again be reduced to the case $C(T)$ after appropriate identifications.
\end{remark}
	




\section{Rank-based empirical tail copulas}
\label{sec:emptailcop}


Since we do not know the marginal  distribution functions $F_t$, we cannot compute $S_{n,\vc{t}}$ in Eq.~\eqref{eq:Snt} from the data. Replacing $F_t$ by the marginal empirical distribution function of the  random variables $\xi_1(t), \ldots, \xi_n(t)$ yields a rank-based estimator of the  tail copula. Let 
\begin{equation}\label{eq:ranks}
	\operatorname{rank}_{n,i}(t)
	= \sum_{\alpha = 1}^n \1\{ \xi_\alpha(t) \le \xi_i(t) \}
\end{equation}
denote the rank of $\xi_i(t)$ among $\xi_1(t), \ldots, \xi_n(t)$.
The empirical tail copula is defined as 
\begin{equation}
\label{eq:Rntx}
	\hat{R}_{n,\vc{t}}(\vc{x})
	=
	\frac{1}{k} \sum_{i=1}^n \1 \left\{
		\forall j = 1, \ldots, d:
		\operatorname{rank}_{n,i}(t_j) > n - k x_j + 1
	\right\}
\end{equation}
for $d$-tuples of coordinates $\vc{t} \in T^d$ and for $\vc{x} \in [0, \infty)^d$.
The empirical tail copula is a rank-based version of the random function $(\vc{t}, \vc{x}) \mapsto  G_{n,\vc{t}}(\vc{x})$ in Eq.~\eqref{eq:Gntx}.
Using $E_{n,\vc{t}}$ in Eq.~\eqref{eq:Entx} for centering, the penultimate  empirical tail copula process is defined as
\[
	\hat{W}_{n}(\vc{v})
	= \sqrt{k} \{ \hat{R}_{n,\vc{t}}(\vc{x}) - E_{n,\vc{t}}(\vc{x}) \},
	\qquad \vc{v} = (\vc{t}, \vc{x}) \in \mathcal{V}.
\]

For $\vc{t} \in T^d$ and $j \in \{1,\ldots,d\}$, let $\dot{R}_{\vc{t},j}(\vc{x})$ denote the partial derivative of $R_{\vc{t}}$ at $\vc{x}$ with respect to $x_j > 0$, provided the partial derivative exists. The standard deviation semimetric $\rho_2$ on $\mathcal{V}$ in Eq.~\eqref{eq:rho2} induces a semimetric $\rho_{2,T}$ on $T$ via
\begin{equation}
\label{eq:rho2T}
	\rho_{2,T}(s, t)
	= \rho_2((s, 1), (t, 1))
	= [2\{1-R_{s,t}(1, 1)\}]^{1/2},
	\qquad (s, t) \in T^2.
\end{equation}
The semimetric space $(T, \rho_{2,T})$ is totally bounded since $(\mathcal{V}, \rho_2)$ is totally bounded by Theorem~\ref{thm:main}. Note that this means that for every $\eps > 0$, there exists a finite set $T(\eps) = \{t_1,\ldots,t_n\} \subset T$ such that for every $t \in T$ there exists $t_j \in T(\eps)$ such that $R_{t,t_j}(1,1) \ge 1 - \eps$. The semimetric $\rho_{2,T}$ is a metric if and only if $R_{s,t}(1, 1) < 1$ whenever $s \ne t$. The (semi)metric space $(T, \rho_{2,T})$ is complete if for every sequence $(t_n)_n$ in $T$ with the property that $\lim_{n \to \infty} \inf_{k,\ell \ge n} R_{t_k,t_\ell}(1, 1) = 1$ there exists $t \in T$ such that $\lim_{n \to \infty} R_{t_k,t}(1, 1) = 1$.

\begin{condition}
	\label{cond:Rdot}
	The  tail copulas satisfy the following properties:
	\begin{enumerate}[(R1)]
		\item 
		 $R_{s,t}(1, 1) < 1$ for all $(s, t) \in T^2$ such that $s \ne t$;
		\item 
		the metric space $(T, \rho_{2,T})$ is complete;
		\item 
		for all $d = 1, \ldots, D$, for all $j = 1, \ldots, d$, and for all $\vc{t} \in T^d$ such that $t_{\alpha} \ne t_{\beta}$ for $\alpha \ne \beta$, the partial derivative $\dot{R}_{\vc{t},j}$ exists and is continuous on the set of $\vc{x} \in [0, \infty)^d$ such that $x_j > 0$.
	\end{enumerate}
\end{condition}

If $x_j = 0$, then $\dot{R}_{\vc{t},j}(\vc{x})$ is defined as the right-hand partial derivative, which exists since $R_{\vc{t}}$ is a concave function on $[0, \infty)^d$.

We will  need to control the difference between the penultimate tail copula $E_{n,\vc{t}}(\vc{x})$ and the tail copula $R_{\vc{t}}(\vc{x})$ itself. Define $E_{n,\vc{t}}(\vc{x}) = 0 = R_{\vc{t}}(\vc{x})$ if one or more coordinates of $\vc{x}$ are negative.

\begin{condition}
	\label{cond:bias}
	For all $d = 1, 2, \ldots, D$ and all $K, M > 0$, as $n \to \infty$,
	\begin{equation*}
		\sup_{\substack{\vc{t} \in T^d, \vc{x} \in [0, M]^d\\\vc{h} \in [-K, K]^d}}
		\Bigl\lvert
			\sqrt{k} \left\{
				E_{n,\vc{t}} \left(\vc{x} + \tfrac{1}{\sqrt{k}}\vc{h} \right)
				- E_{n,\vc{t}}(\vc{x})
			\right\}
			-
			\sqrt{k} \left\{
				R_{\vc{t}} \left(\vc{x} + \tfrac{1}{\sqrt{k}}\vc{h} \right)
				- R_{\vc{t}}(\vc{x})
			\right\}
		\Bigr\rvert
		\to 0.
	\end{equation*}
\end{condition}



It is notationally  convenient, natural, and not  restrictive to consider only tuples of points in $T^d$ with different coordinates. For given $d = 1, 2, \ldots$ and given $M > 0$, let $\mathcal{V}_d'$ be the set of points $(\vc{t},\vc{x}) \in T^d \times [0, M]^d$ such that all $d$ coordinates of $\vc{t}$ are different. As in Theorem~\ref{thm:main} we fix $D=1,2, \ldots$, and write   $\mathcal{V}' = \bigcup_{d=1}^D \mathcal{V}_d'$.


\begin{theorem}
	\label{thm:main:2}
	Consider the setting and conditions of Theorem~\ref{thm:main}. Assume in addition that Conditions~\ref{cond:Rdot} and~\ref{cond:bias} hold. Then  for all $d = 1, 2, \ldots, D$,
	\begin{equation}
	\label{eq:thm:main:2:expansion}
		\sup_{\vc{v} \in \mathcal{V}_d'} \, \Biggl\lvert
			\hat{W}_{n}(\vc{v}) - W_{n}(\vc{v}) + \sum_{j=1}^d \dot{R}_{\vc{t},j}(\vc{x}) \, W_{n}(t_j, x_j)
		\Biggr\rvert
		\stackrel{p}{\to} 0, \qquad n \to \infty.
	\end{equation}
	As a consequence, we have the  weak convergence
	\[
		\hat{W}_{n} \dto \hat{W}, \qquad n \to \infty
	\]
	in the space $\ell^\infty(\mathcal{V}')$, where
	\[
		\hat{W}(\vc{t}, \vc{x}) = W(\vc{v}) - \sum_{j=1}^d \dot{R}_{\vc{t},j}(\vc{x}) \, W(t_j, x_j)
	\]
	and $W$ is the centered Gaussian process in Theorem~\ref{thm:main}. Almost all trajectories of $\hat{W}$ are uniformly $\hat \rho_2$-continuous on $\mathcal{V}'$,
where
\[
	\hat{\rho}_2\bigl((\vc{s},\vc{x}),(\vc{t},\vc{y})\bigr)
	=
	\left( \expec[ \{\hat{W}(\vc{s},\vc{x}) - \hat{W}(\vc{t},\vc{y})\}^2 ] \right)^{1/2}
\]
is the    standard deviation semimetric based on $\hat W$.
\end{theorem}

Condition~\ref{cond:bias} controls the difference between the local increments of $E_{n,\vc{t}}$ and those of $R_{\vc{t}}$. The following condition is stronger and controls the global rate of convergence of $E_{n,\vc{t}}$ to $R_{\vc{t}}$.

\begin{condition}
	\label{cond:bias2}
	For all $M > 0$ and for all $d = 1, 2, \ldots, D$, we have
	\[
		\sup_{\vc{v} \in \mathcal{V}_d} \sqrt{k} \left\lvert
			E_{n,\vc{t}}(\vc{x}) - R_{\vc{t}}(\vc{x})
		\right\rvert \to 0, \qquad n \to \infty.
	\]
\end{condition}

This stronger condition allows us to center $\hat{R}_{n,\vc{t}}(\vc{x})$ with the target $R_{\vc{t}}(\vc{x})$ instead of $E_{n,\vc{t}}(\vc{x})$: the empirical tail copula process
is defined by
\begin{equation}
\label{eq:EmpTailCopProc}
	\sqrt{k} \{\hat{R}_{n,\vc{t}}(\vc{x}) - R_{\vc{t}}(\vc{x})\}, \qquad \vc{v} = (\vc{t}, \vc{x}) \in \mathcal{V}.
\end{equation}
The following corollary to Theorem \ref{thm:main:2} is the main result of the paper.

\begin{corollary}
	\label{cor:main2}
	If, in Theorem~\ref{thm:main:2}, we replace Condition~\ref{cond:bias} by the stronger Condition~\ref{cond:bias2}, then all conclusions in that theorem hold, and they continue to hold if the penultimate empirical tail copula process $\hat{W}_n$ is replaced by the empirical tail copula process in Eq.~\eqref{eq:EmpTailCopProc}.
\end{corollary}

\begin{remark}[Stable tail dependence function] \label{rem:stdf}
Fix  $D=1,2, \ldots \, .$ For $(\vc{t}, \vc{x}) \in T^D \times [0, \infty)^D$,   the stable tail dependence function (stdf) is given by
	\[
		l_{\vc{t}}(\vc{x})
		=
		\lim_{s \downarrow 0}
		s^{-1} \, \prob[ U_1(t_1) \le sx_1 \mbox { or } \ldots \mbox { or }  U_1(t_D) \le sx_D ].
	\]
It is well-known that the stdf characterizes multivariate tail dependence, see, e.g.,  \citet{einmahl+k+s:2012}.
Since  the definition of  $l_{\vc{t}}(\vc{x})$ on $ T^D \times [0, \infty)^D$ involves the probability of a finite union of events it  can be obtained from all the tail copulas $R_{\vc{t}}$ on  $T^d, d=1, \ldots, D$, through the inclusion-exclusion principle, and the same relation holds between the rank-based estimator
\[
	\hat{l}_{n,\vc{t}}(\vc{x})
	=
	\frac{1}{k} \sum_{i=1}^n \1 \left\{
		\exists j = 1, \ldots, D:
		\operatorname{rank}_{n,i}(t_j) > n - k x_j + 1
	\right\}
\]
 of the stdf and all the $\hat R_{n,\vc{t}}$ \citep{chiapino+s+s:2019}. Hence we obtain immediately the result corresponding to Corollary 2 (in particular the weak convergence) for the empirical stdf process
$$\sqrt{k} \{\hat l_{n,\vc{t}}(\vc{x})-l_{\vc{t}}(\vc{x})\}, \qquad \vc{v} = (\vc{t}, \vc{x}) \in \mathcal{V}_{D}'.$$  That result would have been  another way to present the main result of this paper, but for notational ease we chose to focus on tail copulas.
\end{remark}

\begin{remark}[Challenges of the proof]
	\label{rem:challenges}
	The proof of Theorem~\ref{thm:main:2} is particularly challenging because  the partial derivatives $\dot{R}_{\vc{t},j}$ are not uniformly equicontinuous since some of the coordinates of $\vc{t}$ can be arbitrarily close together. This problem arises  for functional data only and asks for a novel approach.  It is easily illustrated for $d=2$. When $t_1$ is fixed and $t_2$ approaches $t_1$, then  $R_{t_1,t_2} $ approaches the non-differentiable comonotone tail copula $R_{t_1,t_1} $, given by $R_{t_1,t_1} (x_1,x_2)=x_1\wedge x_2$.
	Now consider an arbitrarily small neighborhood of the point $(x,x)$. Then if $t_2$ approaches $t_1$, the partial derivative $\dot R_{t_1, t_2,1}$ takes, in this neighborhood, values, that range from about 0 to about 1.
	We solve this problem essentially by taking directional derivatives  along the vector (1,1), that is, along the diagonal. When $d$ is arbitrary, these directional derivatives in the direction of a lower dimensional 1-vector, can be expressed in the ``smooth" partial derivatives with respect to the remaining coordinates and the tail copula itself, which opens the door towards the solution. This program is carried out in the treatment of the term $\Delta_{n,3}$ in the proof, more precisely in Case~II starting on page~\pageref{page:caseII}. In order to make this approach work various interesting lemmas on tail copulas and their partial derivatives are derived in Appendix~\ref{app:aux}.
\end{remark}

\begin{remark}[Empirical copulas]
	Taking $k=n$ instead of $k/n\to 0$, one could also consider empirical copulas rather than empirical tail copulas.  It is known from the multivariate case that the limiting process will be different then, because all $n$ data are taken into account which leads to a tied-down process, a bridge, rather than a Wiener process, as in Theorem~\ref{thm:main}.  This goes back to \cite{r73}; see \cite{segers2012} and the references therein.
To prove in the functional setting an analogue of Theorem \ref{thm:main:2}, however, is a new research project, since the proof techniques used here cannot be transferred, mainly   because the partial derivatives of a copula enjoy fewer regularity properties than those of a tail copula, since a copula is in general  not concave and not homogeneous.
\end{remark}

We conclude this section with two examples for which we work out the conditions in some detail. From these, more examples can be generated by transforming the trajectories $t \mapsto \xi(t)$ coordinate-wise to $t \mapsto f(t, \xi(t))$, where $f : T \times \reals \to \reals$ is continuous and $z \mapsto f(t, z)$ is increasing for every $t \in T$. Such a transformation only changes the marginal distributions $F_t$ but leaves the uniformized process $U_t = 1 - F_t(\xi(t))$ unaffected.

\begin{example}[Smith model]\label{exs}
	Let $T=[0,1]^r$ for some $r=1, 2, \ldots$ and let $\{\vc{S}_i, Y_i\}_{i \ge 1}$ be the points of a unit-rate homogeneous Poisson process on $\reals^r\times (0,\infty)$. Consider the process
	\[
		\xi(\vc{u})=\max_{i = 1, 2, \ldots} Y_i^{-1} f(\vc{S_i}-\vc{u}),
		\qquad \vc{u} \in [0,1]^r,
	\]
	where
	\[
		f(\vc{s})
		= (2\pi)^{-r/2}|\Sigma|^{-1/2}
		\exp\{-\tfrac{1}{2} \vc{s}' \Sigma^{-1}\vc{s}\},
		\qquad \vc{s}\in \reals^r,
	\]
	is a centered normal density, with $\Sigma$ an invertible $r \times r$  covariance matrix. The process $\xi$ follows the so-called Smith model, after \citet{smith:1990}. It has been applied and extended in \citet{coles:1993}, \citet{schlather:2002} and \citet{kabluchko+s+dh:2009}, among others, and is a special case of the Brown--Resnick process. The process $\xi$ is  a continuous, stationary, max-stable process with unit-Fr\'echet marginals $F_{\vc{u}}(z)=\exp(-1/z)$, $z>0$.

	The finite-dimensional distributions of $\xi$ can be computed via the probability that the Poisson process with points $(\vc{S}_i, Y_i)$ misses a certain subset of $\reals^r \times (0, \infty)$. Let $\vc{S}$ be a random vector with density $f$, i.e., a $N_r(\vc{0}, \Sigma)$ random vector, and define a stochastic process on $\reals^r$ by
	\[
		B(\vc{u}) = \exp \left(
			\vc{S}' \Sigma^{-1} \vc{u}
			- \tfrac{1}{2} \vc{u}' \Sigma^{-1} \vc{u}
		\right),
		\qquad \vc{u} \in \reals^r.
	\]
	Note that $B(\vc{u})$ is a log-normal random variable with unit expectation. We find, for $(\vc{u}_1, \ldots, \vc{u}_d) \in T^d$ and $(z_1, \ldots, z_d) \in (0, \infty)^d$, that
	\begin{align*}
		- \ln \prob\{\forall j = 1, \ldots, d : \xi(\vc{u}_j) \le z_j\}
		&=
		\int_{\vc{s} \in \reals^r} \max_{j=1,\ldots,d} \{z_j^{-1} f(\vc{s}-\vc{u}_j)\} \, \diff \vc{s} \\
		&=
		\expec \left[
			\max_{j=1,\ldots,d} z_j^{-1} B(\vc{u}_j)
		\right].
	\end{align*}
	The second identity follows from computing $f(\vc{s}-\vc{u}_j) / f(\vc{s})$ inside the maximum in the integral. The finite-dimensional distributions of $\xi$ are multivariate H\"usler--Reiss distributions \citep{husler+r:1989}. The last expression represents the finite-dimensional distributions via a $D$-norm \citep{falk:2019}.
	
	The tail copula of $\xi$ can be obtained from the finite-dimensional distributions via the inclusion--exclusion formula and the minimum-maximum identity: for $(\vc{u}_1,\ldots,\vc{u}_d) \in T^d$ and $\vc{x} \in [0, \infty)^d$
	\begin{align}
	\label{eq:R:Smith}
		R_{\vc{u}_1,\ldots,\vc{u}_d}(\vc{x})
		&=
		\int_{\vc{x} \in \reals^r} \min_{j=1,\ldots,d} \{x_j f(\vc{s}-\vc{u}_j)\} \, \diff \vc{s} \\
	\nonumber
		&=
		\expec \left[
			\min_{j=1,\ldots,d} x_j B(\vc{u}_j)
		\right].
	\end{align}
	The partial derivative of $R_{\vc{u}_1,\ldots,\vc{u}_d}(\vc{x})$ with respect to $x_j$ in a point $\vc{x}$ such that $x_j$ is positive is
	\[
		\dot{R}_{\vc{u}_1,\ldots,\vc{u}_d;j}(\vc{x})
		=
		\expec \left[ B(\vc{u}_j) \1 \left\{
			x_j B(\vc{u}_j) = \min_{\alpha=1,\ldots,d} x_\alpha B(\vc{u}_\alpha) \right\}
		\right]
	\]
	By the dominated convergence theorem, this expression is continuous in $\vc{x}$ such that $x_j > 0$.
	
	The bivariate tail copula can be further computed from the integral formula~\eqref{eq:R:Smith} via some elementary algebra: for $(\vc{u},\vc{v},x,y) \in T^2 \times (0, \infty)^2$,
	\[
		R_{\vc{u},\vc{v}}(x, y)
		= x \, \bar{\Phi}\left(
			\tfrac{a(\vc{u},\vc{v})}{2} + \tfrac{1}{a(\vc{u},\vc{v})} \ln(x/y)
		\right)
		+ y \, \bar{\Phi}\left(
			\tfrac{a(\vc{u},\vc{v})}{2} + \tfrac{1}{a(\vc{u},\vc{v})} \ln(y/x)
		\right)
	\]
	where $a(\vc{u},\vc{v}) = \{(\vc{u}-\vc{v})'\Sigma^{-1}(\vc{u}-\vc{v})\}^{1/2} $, while $\bar{\Phi} = 1 - \Phi$ and $\Phi$ is the standard normal distribution function. In particular, $R_{\vc{u},\vc{v}}(1, 1) = 2 \bar{\Phi}(a(\vc{u},\vc{v})/2)$, which is less than $1$ as soon as $\vc{u} \ne \vc{v}$. The standard deviation semimetric $\rho_{2,T}$ in \eqref{eq:rho2T} is equivalent to the Euclidean metric on $T$.
	
	We verify Condition~\ref{cond:bias2}. By the inclusion--exclusion formula, $\prob\{\forall j = 1, \ldots, d : \xi(\vc{u}_j) > z_j \}$ is a linear combination of expressions of the form $\exp[ - \expec\{\max_{j\in A} z_j^{-1} B(\vc{u}_j)\} ]$, where $A$ ranges over the non-empty subsets of $\{1,\ldots,d\}$. It follows that $E_{n,\vc{u}_1,\ldots,\vc{u}_d}(\vc{x})$ in \eqref{eq:Entx} is a linear combination of expressions of the form
	\[
		\exp( - \expec[\max_{j\in A} \{-\ln(1-\tfrac{k}{n}x_j) B(\vc{u}_j)\}] ).
	\]
	A Taylor expansion of the functions $h \mapsto \exp(-h)$ and $h \mapsto -\ln(1-h)$ as $h \to 0$ yields that the convergence rate of $E_{n,\vc{u}_1,\ldots,\vc{u}_d}(\vc{x})$ to $R_{\vc{u}_1,\ldots,\vc{u}_d}(\vc{x})$ is $O(k/n)$ uniformly in $T^d \times [0, M]^d$, for $M > 0$. It follows that Condition~\ref{cond:bias2} is fulfilled as soon as $k^{3/2} = o(n)$ as $n \to \infty$.
	
	We have thus verified that the conditions of Corollary~\ref{cor:main2} are satisfied, except for Condition~\ref{cond:cover:balls}. But the latter can be shown to hold for the uniformized version $U$ of $\xi$ using the arguments on page~476 of  \citet{einmahl+l:2006}.

	 Since the process $\xi$ in this example is max-stable, the tail copulas of its finite-dimensional distributions are linked to the Pickands dependence functions of these distributions. Using this additional knowledge, the latter functions can be estimated nonparametrically at rate $O_p(1/\sqrt{n})$, see for instance \citet{gudendorf+s:2012}. Our estimator is constructed to be valid for a much broader class of processes for which the tail copula only arises in the limit, and this explains its convergence rate $O_p(1/\sqrt{k})$.
\end{example}

\begin{example}[Pareto process]\label{exp}
Let $T=[0,1]$. Let $Y$ be a standard Pareto random variable, i.e., $\prob(Y \leq x) = 1-1/x$, for $x\geq 1$, and let $W'$ be a standard Wiener process on $[0,1]$.
Consider $B(t)=\exp\{W'(t)-t/2\}$, $t \in [0,1]$, a geometric Brownian motion. We have $B(t)>0$, $\expec\{B(t)\}=1$ for all $t \in [0,1]$,
and $\expec\{\sup_{t \in [0, 1]} B(t)\} < \infty$. Assume  $Y$ and $W'$ are independent.
Define
\[
	\xi(t) = Y B(t), \qquad t \in [0,1],
\]
a process introduced in \citet{gdhp:2004}; see \citet{ferreira+dh:2014} and \citet{dombry+r:2015} for related processes.
The process $\xi$ is continuous
and its finite-dimensional distributions can be found via Fubini's theorem, giving
\[
	\prob[\forall j = 1, \ldots, d: \xi(t_j) \ge y/x_j]
	= y^{-1} \, \expec[\min\{ y, x_1 B(t_1), \ldots, x_d B(t_d)\}]
\]
for $y, x_1, \ldots, x_d \in [0, \infty)$ and $t_1, \ldots, t_d \in T$. It follows that the tails of the margins of $F_t$ are very close to standard Pareto,
\[
	\forall m > 0, \qquad
	\sup_{t \in T} \left| 1 - F_t(y) - 1/y \right| = o(y^{-m}),
	\qquad y \to \infty,
\]
and that the tail copula is
\[
	R_{\vc{t}}(\vc{x}) = \expec[\min\{x_1 B(t_1), \ldots, x_d B(t_d)\}]
\]
for $d = 1, 2, \ldots$ and $(\vc{t},\vc{x}) \in T^d \times [0, \infty)^d$.
Moreover, 
\[
	\forall m > 0, \qquad
	\sup_{\vc{v} \in \mathcal{V}_d}
	\left| E_{n,{\vc{t}}}(\vc{x}) - R_{\vc{t}}(\vc{x}) \right|
	= o((k/n)^m), \qquad n \to \infty.
\]
Condition~\ref{cond:bias2} is thus satisfied as soon as $k^{1+\eta} = o(n)$ for some $\eta > 0$. If all coordinates $t_1, \ldots, t_d$ are distinct, the partial derivative of $R_{\vc{t}}(\vc{x})$ with respect to $x_j$ is
\[
	\dot{R}_{\vc{t},j}(\vc{x})
	=
	\expec\left[
		B_j(t_j) \,
		\1 \left\{
			x_j B(t_j) = \min_{\alpha=1,\ldots,d} x_\alpha B(t_\alpha)
		\right\}
	\right],
\]
which is continuous in points $\vc{x} \in [0, \infty)^d$ such that $x_j > 0$, thanks to the dominated convergence theorem. The bivariate tail copula is
	\[
R_{s,t}(x, y)
= x \, \bar{\Phi}\left(
\tfrac{\sqrt{|s-t|}}{2} + \tfrac{1}{\sqrt{|s-t|}} \ln(x/y)
\right)
+ y \, \bar{\Phi}\left(
\tfrac{\sqrt{|s-t|}}{2} + \tfrac{1}{\sqrt{|s-t|}} \ln(y/x)
\right)
\]
for $(s,t,x,y) \in [0,1]^2 \times (0, \infty)^2$. In particular, $R_{s,t}(1,1) = 2 \bar{\Phi}(\sqrt{|s-t|}/2)$, so that the semimetric $\rho_{2,T}$ on $T$ is equivalent to the absolute value metric.

Again Condition~\ref{cond:cover:balls} can be shown to hold for the uniformized version $U$ of $\xi$ using  arguments in \citet[pages 479--480]{einmahl+l:2006}.
\end{example}

%

\section{Testing stationarity}
\label{sec:testStat}

Under the conditions of Corollary~\ref{cor:main2}, the empirical tail copula process in \eqref{eq:EmpTailCopProc} converges weakly in the space $\ell^\infty(\mathcal{V}')$ to the centred, Gaussian process $\hat{W}$ described in Theorem~\ref{thm:main:2}. Here, we apply the result to implement a novel test for spatial stationarity of the tail copula in case $T = [0, 1]$. We evaluate the finite-sample performance of the test for the Smith model and the Pareto process in Examples~\ref{exs} and~\ref{exp} respectively by means of Monte Carlo simulations.

 The functional data $\xi_i$ are said to be tail copula stationary if
\begin{equation}
\label{eq:Rstationary}
	R_{t_1+h,\ldots,t_d+h}(\vc{x}) = R_{t_1,\ldots,t_d}(\vc{x})
\end{equation}
for any $d = 2, 3, \ldots$, any $\vc{t} \in [0, 1]^d$, any $\vc{x} \in [0, \infty)^d$ and any $h \in [-1, 1]$ such that $t_j + h \in [0, 1]$ for all $j = 1, \ldots, d$. Given 
a random sample $\xi_1,\ldots,\xi_n$ in $C([0, 1])$, we will use the empirical tail copula $\hat{R}_{n,\vc{t}}$ in \eqref{eq:Rntx} at threshold parameter~$k = k_n$ as in Condition~\ref{cond:kn} to test whether \eqref{eq:Rstationary} holds for $d = 2$ and $(x_1, x_2) = (1, 1)$. Further, we assume that we only observe $\xi_i(t)$ for $t \in \{r / N : r = 0, 1, \ldots, N\}$ for some integer $N \ge 3$.

For a positive integer $\Delta < N/2$ and for $r_0 \in \{\Delta, \ldots, N-\Delta\}$, put
\begin{equation}
\label{eq:InN}
	\hat{I}_n^{(N)}(r_0/N)
	= \frac{1}{2\Delta} \sum_{r : 1 \le |r - r_0| \le \Delta}
	\hat{R}_{n;r/N, r_0/N}(1, 1)
\end{equation}
and  define the test statistic 
\begin{equation}
\label{eq:DnN}
	D_n^{(N)} = \sqrt{k} \left[
	\max_{r_0 = \Delta, \ldots, N-\Delta} \hat{I}_n^{(N)}(r_0/N)
	-
	\min_{r_0 = \Delta, \ldots, N-\Delta} \hat{I}_n^{(N)}(r_0/N)
	\right].
\end{equation}
Assume the conditions of Corollary~\ref{cor:main2}. Under the null hypothesis of stationarity and for fixed $N$ and $\Delta$, the asymptotic distribution of the test statistic is
\begin{equation}
\label{eq:DnNDN}
	D_n^{(N)} \dto D^{(N)}
	=
	\max_{r_0 = \Delta, \ldots, N - \Delta} V^{(N)}(r_0/N)
	-
	\min_{r_0 = \Delta, \ldots, N - \Delta} V^{(N)}(r_0/N),
	\qquad n \to \infty,
\end{equation}
where, for $r_0 \in \{\Delta, \ldots, N-\Delta\}$,
\begin{equation}
\label{eq:VN}
	V^{(N)}(r_0/N)
	=
	\frac{1}{2\Delta} \sum_{r : 1 \le |r-r_0| \le \Delta}
	\hat{W}(r/N,r_0/N; 1, 1).
\end{equation}
and $\hat{W}$ is the limit process in Theorem~\ref{thm:main:2}. Corollary~\ref{cor:main2} even allows us to find the limit of $D_n^{(N)}$ in case $N = N_n \to \infty$ and $\Delta = \Delta_n \to \infty$ in such a way that $\Delta/N \to \delta \in (0, 1/2)$.  This is the more relevant and more interesting  case when considering functional data.  If the function $(s, t) \mapsto R_{s,t}(1, 1)$ is continuous with respect to the Euclidean metric on $[0, 1]^2$,   then
\[
	D_n^{(N)} \dto D
	=
	\sup_{t \in [\delta, 1-\delta]} V(t) - \inf_{t \in [\delta, 1-\delta]} V(t)
\]
where $V(t_0) = \frac{1}{2\delta} \int_{t_0-\delta}^{t_0+\delta} \hat{W}(t, t_0; 1, 1) \, \diff t$ for $t_0 \in [\delta, 1-\delta]$. A sketch of the proofs of these statements is given in Appendix~\ref{app:testStat}.

If the tail copula fails to be stationary, we can expect $D_n^{(N)} \to \infty$ in distribution. Therefore we reject the null hypothesis of stationarity for large values of $D_n^{(N)}$. To find the $p$-value of the observed test statistic, we rely on the asymptotic distribution under the null hypothesis  $D^{(N)}$, which is based on
a centered Gaussian random vector of dimension $N-2\Delta+1$. In view of the definition of $\hat{W}$ in Theorem~\ref{thm:main:2}, the covariance matrix of $V^{(N)}$ can be
expressed in terms of the tail dependence coefficients $R_{\vc{t}}(\vc{1})$ for $\vc{t} \in \{0, 1/N, \ldots, 1\}^d$ for $d \in \{2, 3, 4\}$ and the partial derivatives $\dot{R}_{t_1, t_2; j}(1, 1)$ for $(t_1, t_2) \in \{0, 1/N, \ldots, 1\}^2$ and $j \in \{1, 2\}$. These can be estimated consistently, again by Corollary~\ref{cor:main2}; for the partial derivatives use finite differencing at a bandwidth sequence that goes to zero more slowly than $k^{-1/2}$ and the fact that by stationarity, $\dot{R}_{t_1, t_2; j}(1, 1)$ only depends on $t_2 - t_1$. Under the null hypothesis, the covariance matrix of $V^{(N)}$ is Toeplitz, a fact which can be used in its estimation as well. The distribution function of $D^{(N)}$ at the estimated covariance matrix of $V^{(N)}$ can then be calculated in terms of the multivariate normal cumulative distribution function (cdf). For details, we refer to Appendix~\ref{app:testStat}.

We have implemented the test in \textsf{R} \citep{R}, relying on the package \textsf{mnormt} \citep{azzalini+g:2020} for calculating the multivariate normal cdf as detailed in Appendix~\ref{app:testStat}.
We then evaluated the test's finite-sample performance for the Smith model with $\Sigma = 1$ and the Pareto process in Examples~\ref{exs} and~\ref{exp}, respectively, both of which are  tail copula stationary.  For drawing samples from the Smith model, we used the package \textsf{spatialExtremes} \citep{ribatet:2020}. To assess the power of the test against alternatives, we simulated from the same processes but on distorted grids of the form $\{f_\theta(r/N) : r = 0, 1, \ldots, N\} \in [0, 1]$ where $f_{\theta}(t) = 	[(2 - \theta t) t] / [1 - \theta + \sqrt{(2 - \theta t) \theta t + (1 - \theta)^2}]$ for $\theta \in [0, 1]$. Note that $f_0(t) = t$ while $f_1(t) = \sqrt{(2 - t)t}$, the graph of a quarter circle with center $(1, 0)$ and radius $1$.

The results for $N = 20$, $\Delta = 2$, $n = 500$ and $k = 50$ are shown in Figure~\ref{fig:testStat}, with the Smith model on the left and the Pareto process on the right:
\begin{itemize}
	\item The top row shows the probability mass function (pmf) of the integer-valued random variable $2\Delta\sqrt{k} D_n^{(N)}$ under the null hypothesis together with the probability density function (pdf) of the random variable $2\Delta\sqrt{k} D^{(N)}$ with $D^{(N)}$ in \eqref{eq:DnNDN}; please see Appendix~\ref{app:testStat}
	for how we calculated the latter. Note that, despite the bell shape of the pdf, the limit variable is not Gaussian. The pmf is estimated on the basis of $5\,000$ random samples.
	\item The middle row shows PP-plots of the $p$-values based on the estimated covariance matrix of $V^{(N)}$ based on $1\,000$ samples under the null hypothesis.
	\item The bottom row shows the rejection probabilities at significance level $\alpha = 0.05$ at the null hypothesis ($\theta = 0$) and at alternatives ($\theta \in \{0.1, 0.2, \ldots, 1\}$) and this for $\Delta \in \{1, \ldots, 4\}$, based on $1\,000$ samples for each value of $\theta$. The null hypothesis is rejected if the $p$-value calculated using the estimated covariance matrix of $V^{(N)}$ does not exceed   $\alpha$.
\end{itemize}

 The asymptotic theory works well in the sense that  the continuous limit distribution of the test statistic matches its discrete finite-sample distribution quite closely. The $p$-values look reasonably uniform, especially for the Pareto process, despite the complicated setup and the fact that the covariance matrix of $V^{(N)}$ is re-estimated for each sample.  Also for the Smith process the fit is very good for the relevant small values of the significance level.  From the power plots, we see that the test picks up the alternative quite well.  Similar patterns appeared for other choices $n, k, N, \Delta$. Overall, the plots show that distributional approximations based on Corollary~\ref{cor:main2} are very useful already at moderate sample sizes.

\begin{figure}
\begin{center}
\begin{tabular}{@{}c@{}c}
\includegraphics[width=0.49\textwidth]{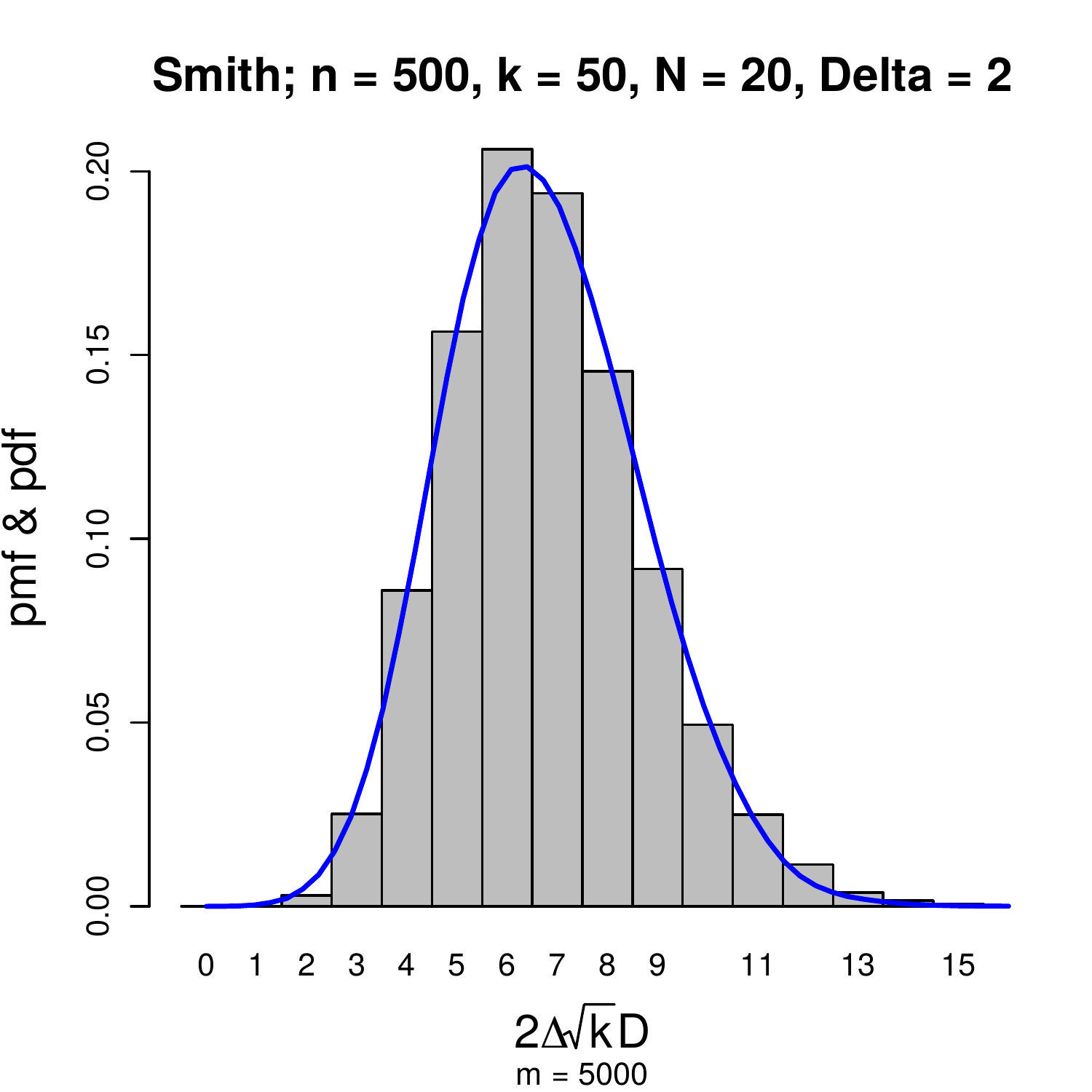}&
\includegraphics[width=0.49\textwidth]{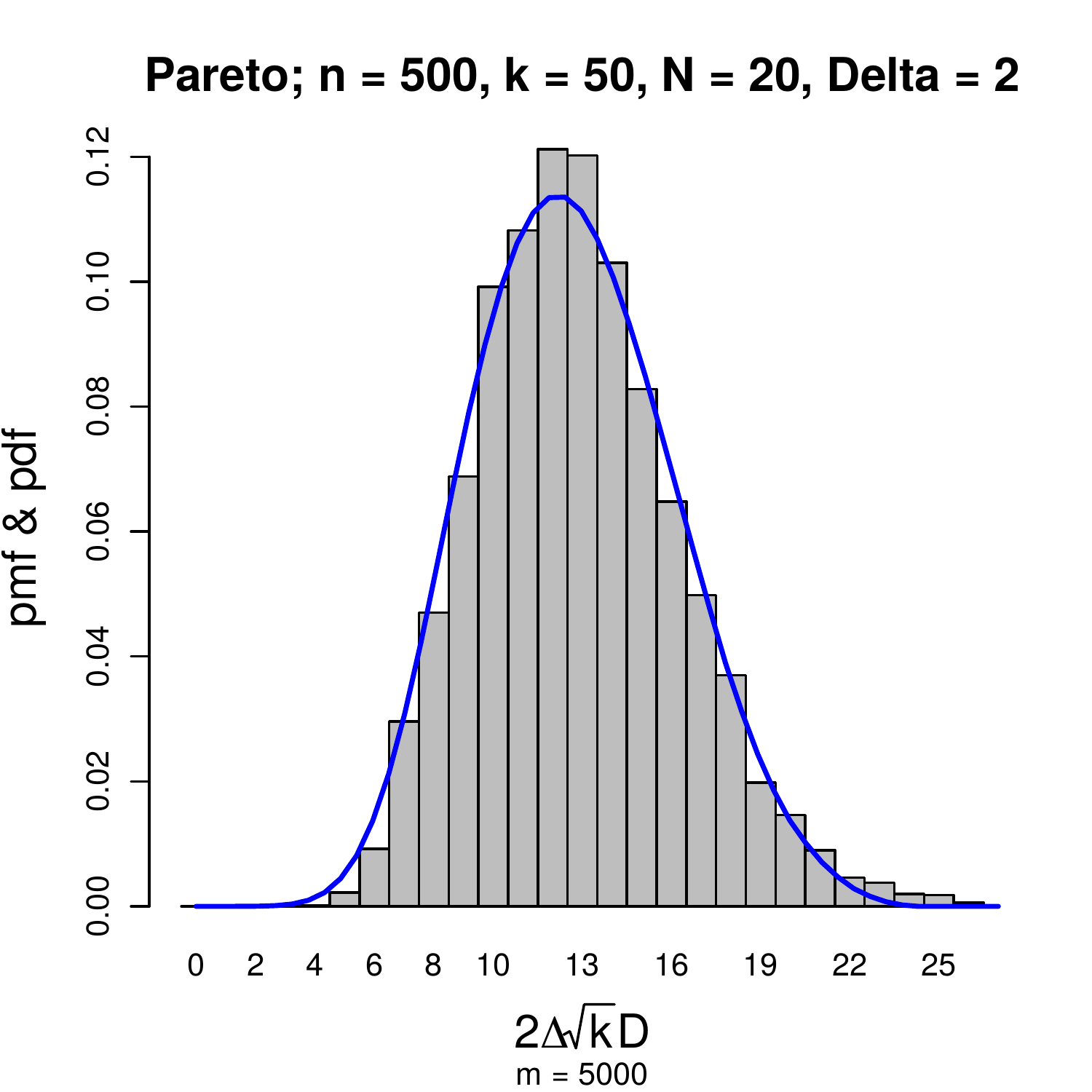}\\
\includegraphics[width=0.49\textwidth]{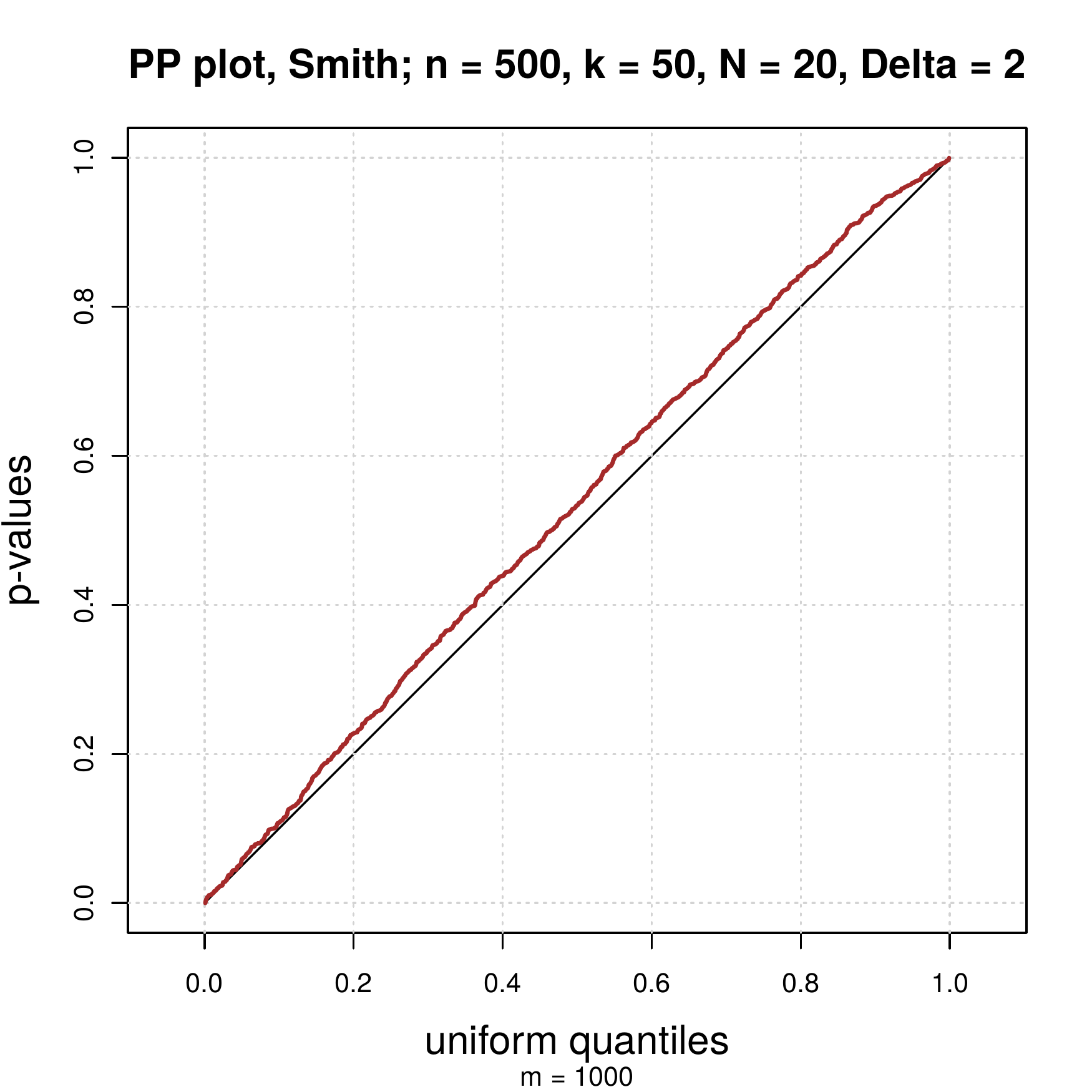}&
\includegraphics[width=0.49\textwidth]{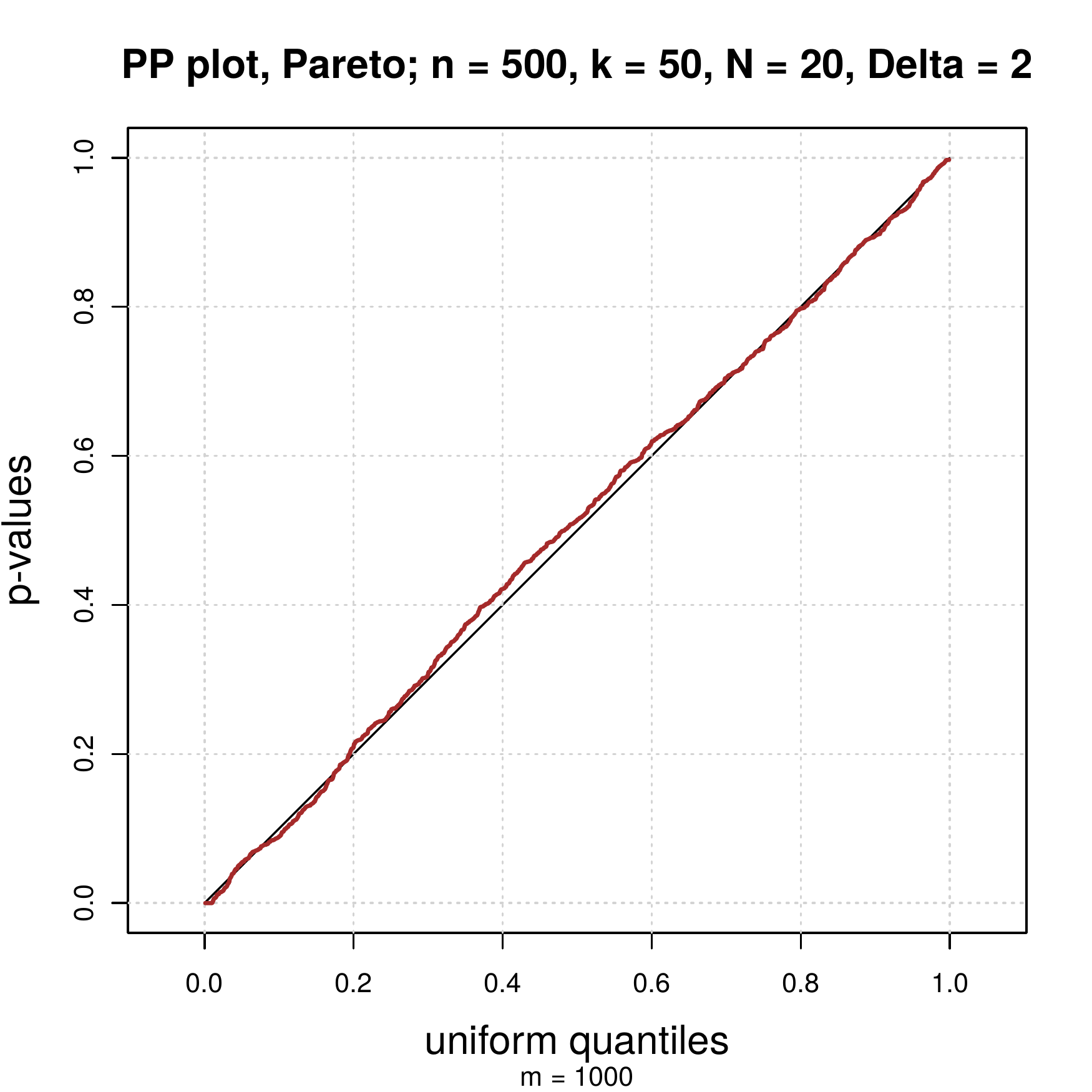}\\
\includegraphics[width=0.49\textwidth]{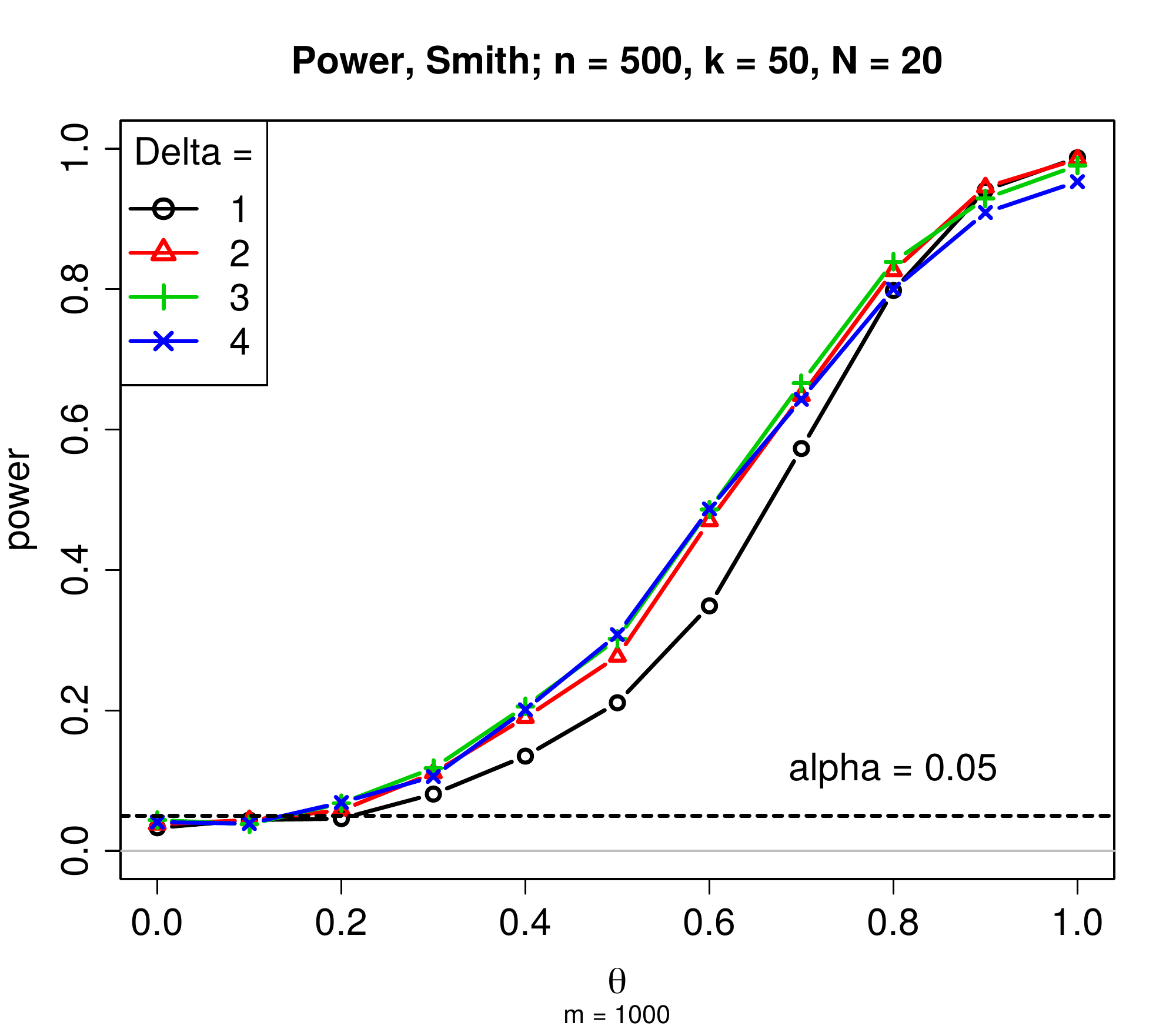}&
\includegraphics[width=0.49\textwidth]{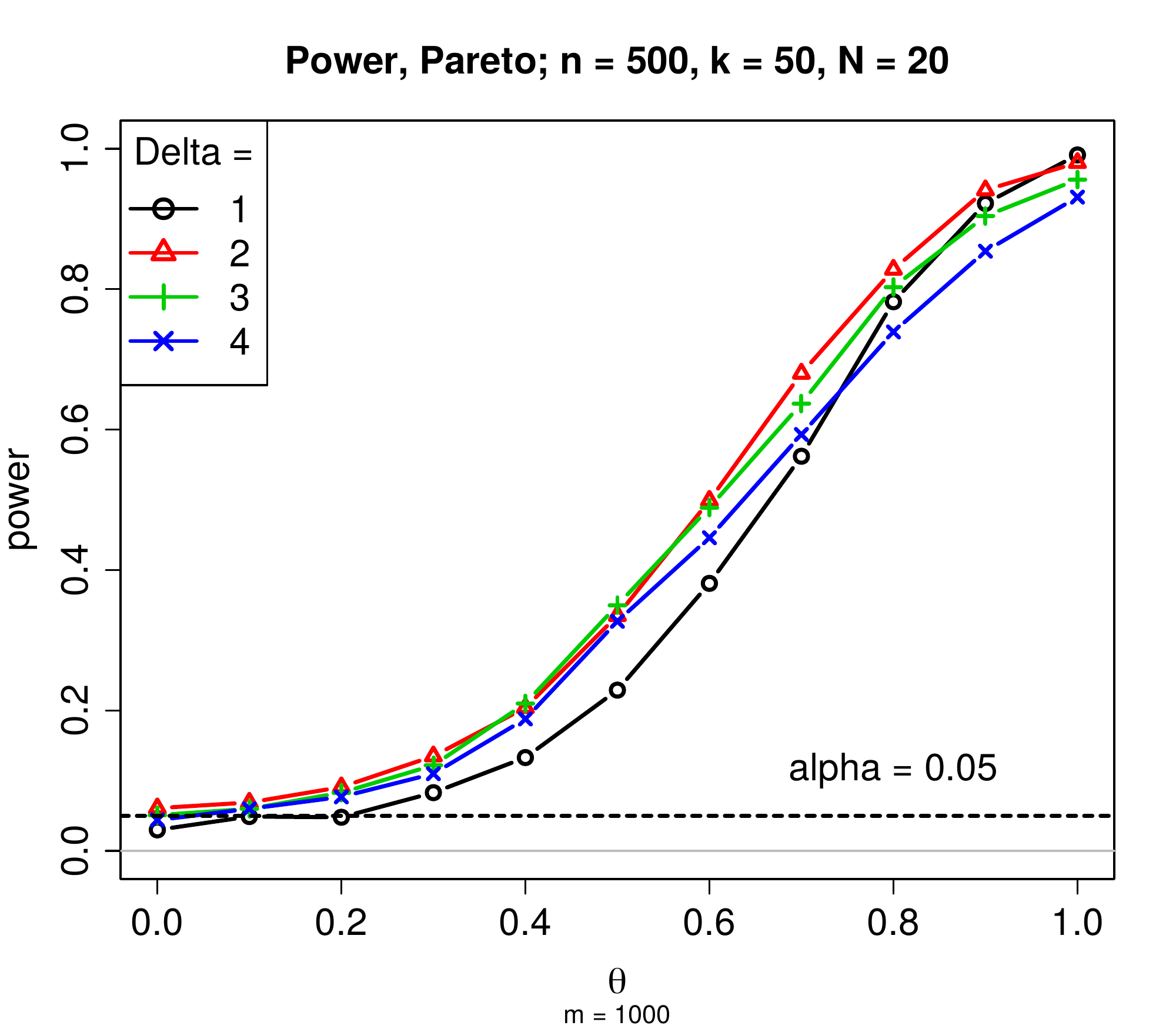}
\end{tabular}
\end{center}
\caption{\label{fig:testStat} Test of Section~\ref{sec:testStat} for tail copula stationarity for stochastic processes on $[0, 1]$ observed on $\{r/N : r = 0, 1, \ldots, N\}$ with $N = 20$. Sample size $n = 500$, threshold parameter $k = 50$. Left: Smith model (Example~\ref{exs}). Right: Pareto process (Example~\ref{exp}). Top: sampling distribution of the discrete test statistic under the null hypothesis together with the density of the limiting variable. Middle: PP-plot of the $p$-values of the test under the null hypothesis. Bottom: Powers at significance level $\alpha = 0.05$ against alternatives parametrized by $\theta \in [0, 1]$. Plots based on $5\,000$ samples at the top, $1\,000$ samples in the middle, and $1\,000$ samples per value of $\theta$ at the bottom.}
\end{figure}

\section{Proofs for the results of Section~\ref{sec:tailempproc}}
\label{sec:proofs:tail}

We first collect some results for the proof of Theorem~\ref{thm:main}.

\begin{lemma}
	\label{lem:tail}
	On some non-empty set $T$, let $U = (U(t))_{t \in T}$ be a stochastic process with uniform-$(0, 1)$ margins. Assume that there exists a finite covering $T = \bigcup_{j=1}^N T_j$ of $T$ and a scalar $\lambda > 1$ such that
	\begin{equation}
	\label{eq:lem:tail}
		\forall j = 1, \ldots, N, \qquad
		\liminf_{x \downarrow 0}
		\prob \left[
			\sup_{t \in T_j} U(t) \le \lambda x
			\, \left\vert \,
			\inf_{t \in T_j} U(t) \le x
		\right. \right]
		> 0.
	\end{equation}
	Then $\prob[ \inf_{t \in T} U(t) \le x] = O(x)$ as $x \downarrow 0$.
\end{lemma}

\begin{proof}
	Write $A_j = \inf_{t \in T_j} U(t)$. Since
	\[
		\prob \left[ \inf_{t \in T} U(t) \le x \right]
		=
		\prob \left[ \min_{j=1,\ldots,T} A_j \le x \right]
		\le
		\sum_{j=1}^N \prob \left[ A_j \le x \right],
	\]
	it is sufficient to show that $\prob[ A_j \le x ] = O(x)$ as $x \downarrow 0$ for every $j = 1, \ldots, N$. Write $B_j = \sup_{t \in T_j} U(t)$ and note that $\prob[B_j \le y] \le \prob[U(t) \le y] \le y$ for $y \ge 0$, where $t \in T_j$ is arbitrary. Then
	\begin{align*}
		\prob\left[ A_j \le x \right]
		&=
		\prob\left[ A_j \le x, B_j \le \lambda x \right]
		+
		\prob\left[ A_j \le x, B_j > \lambda x \right] \\
		&\le
		\lambda x
		+
		\prob\left[ A_j \le x \right]
		\prob\left[ B_j > \lambda x \mid A_j \le x \right].
	\end{align*}
	Solving for $\prob\left[ A_j \le x \right]$ yields
	\[
		\prob\left[A_j \le x \right]
		\le
		\frac%
			{\lambda x}%
			{\prob \left[ B_j \le \lambda x \mid A_j \le x\right]}.
	\]
	For sufficiently small $x$, the denominator on the right-hand side is bounded away from zero by assumption.
\end{proof}

The following theorem is a corollary to Theorem~2.11.9 in \citet{vandervaart+w:1996}.

\begin{theorem}
	\label{thm:tight}
	For each $n$, let $Z_{n,1}, \ldots, Z_{n,n}$ be independent stochastic processes with finite second moments indexed by a set $\mathcal{F}$. Write $\norm{Z_{n,i}}_{\mathcal{F}} = \sup_{f \in \mathcal{F}} \abs{Z_{n,i}(f)}$. Suppose
	\begin{equation}
	\label{eq:tight:1}
		\forall \lambda > 0, \qquad
		\lim_{n \to \infty} \sum_{i=1}^n \expec^* \left[ \norm{Z_{n,i}}_{\mathcal{F}} \1 \{ \norm{Z_{n,i}}_{\mathcal{F}} > \lambda \} \right] = 0,
	\end{equation}
	and suppose that there exists $c > 0$ and, for all sufficiently small $\eps > 0$, a covering $\mathcal{F} = \bigcup_{j=1}^{N_\eps} \mathcal{F}_{\eps,j}$ of $\mathcal{F}$ by sets $\mathcal{F}_{\eps,j}$ such that, for every set $\mathcal{F}_{\eps,j}$ and every $n$,
	\begin{equation}
	\label{eq:ceps2}
		\sum_{i=1}^n \expec^* \left[
			\sup_{f,g \in \mathcal{F}_{\eps,j}} \abs{Z_{n,i}(f) - Z_{n,i}(g)}^2
		\right]
		\le c \eps^2
	\end{equation}
	and, for some $\delta > 0$, we have
	\begin{equation}
	\label{eq:entropy}
		\int_0^{\delta} \sqrt{\log N_\eps} \, \diff \eps < \infty.
	\end{equation}
	Then the sequence $\sum_{i=1}^n \{Z_{n,i} - \expec(Z_{n,i})\}$ is asymptotically tight in $\ell^\infty(\mathcal{F})$ and converges weakly provided the finite-dimensional distributions converge weakly.
\end{theorem}

\begin{proof}
	We show that the conditions of Theorem~2.11.9 in \citet{vandervaart+w:1996} are fulfilled. First, from a given covering, we can easily extract a partition meeting the same requirements by considering successive differences of the covering sets, and hence the sets in the  partition are subsets of the corresponding sets in the  covering.
	
	In comparison to the cited theorem, our partitions do not depend on $n$. As a consequence, the middle one of the three displayed conditions in that theorem is redundant. But then the semimetric $\rho$ on $\mathcal{F}$ does not play any role either and it can be any semimetric that makes $\mathcal{F}$ totally bounded; take for instance the trivial zero semimetric.
	
	Eq.~\eqref{eq:entropy} obviously implies that $\int_0^{\delta_n} \sqrt{\log N_\eps} \, \diff \eps \to 0$ for every $\delta_n \downarrow 0$.
	
	Finally, the upper bound in \eqref{eq:ceps2} has $c \eps^2$ rather than just $\eps^2$ as in the displayed equation just above the statement of Theorem~2.11.9 in~\cite{vandervaart+w:1996}. This makes no difference since we can just take the cover or partition associated to $\eps / \sqrt{c}$. Indeed, the bracketing number $N_{[\,]}(\eps, \mathcal{F}, L_2^n)$ in that theorem is then bounded by $N_{\eps / \sqrt{c}}$, and for sufficiently small $\delta > 0$, the integral $\int_0^{\sqrt{c} \delta} \sqrt{\log N_{\eps / \sqrt{c}}} \, \diff \eps = \sqrt{c} \int_0^\delta \sqrt{\log N_\eps} \, \diff \eps$ is finite.
\end{proof}

\begin{proof}[Proof of Theorem~\ref{thm:main}]
We will apply  Theorem~\ref{thm:tight}  with $\mathcal{F}$  equal to $\mathcal{V}$. 
Define stochastic processes $Z_{n,i}$ on $\mathcal{V}$ by
\[
Z_{n,i}(\vc{v})
= k^{-1/2} \1 \left\{ \forall j = 1, \ldots, d : U_i(t_j) \le \tfrac{k}{n} x_j \right\},
\qquad \vc{v} = (\vc{t}, \vc{x}) \in \mathcal{V}_d,
\]
so that $W_n(\vc{v}) = \sum_{i=1}^n [ Z_{n,i}(\vc{v}) - \expec\{Z_{n,i}(\vc{v})\} ]$.
\smallskip

\emph{Step~1: Finite-dimensional distributions.} ---
The finite-dimensional distributions of $W_n$ converge due to the Lindeberg central limit theorem. Indeed, $W_n$ is centered and the covariance function of $W_n$ converges pointwise to the one of $W$, see Eq.~\eqref{eq:covn}. On the one hand, if $\var\{W(\vc{v})\} = R_{\vc{t}}(\vc{x}) = 0$, then $\lim_{n \to \infty} \var(W_n(\vc{v})) = 0$ too, and thus $W_n(\vc{v})$ converges in probability to $W(\vc{v}) = 0$. On the other hand, if the limit variance is positive, then the Lindeberg condition is trivially fulfilled, since $0 \le Z_{n,i} \le k^{-1/2} \to 0$ as $n \to \infty$, so that for every $\eps > 0$, the indicators $\1 \{ \abs{Z_{n,i}(\vc{v}) - \expec[Z_{n,i}(\vc{v})]} \ge \eps \sqrt{\var W_n(\vc{v})} \}$ in the Lindeberg condition are all equal to zero for sufficiently large $n$.
\smallskip

\emph{Step~2: Checking Eq.~\eqref{eq:tight:1}.} ---
Fix $\lambda > 0$. Since $\sup_{\vc{v} \in \mathcal{V}} \abs{Z_{n,i}(\vc{v})} \le k^{-1/2} \to 0$ as $n \to \infty$ by Condition~\ref{cond:kn}, the indicator variable inside the outer expectation is zero for $n$ sufficiently large, such that $k^{-1/2} \le \lambda$.
\smallskip

\emph{Step~3: Constructing the covering.} ---
Let $\eps > 0$ be small. For each $d = 1, \ldots, D$, we construct a covering of $\mathcal{V}_d = \bigcup_{m=1}^{N_{d,\eps}} \mathcal{V}_d(\eps,m)$ such that, for every covering set $\mathcal{V}_d(\eps,m)$, we have
\[
\sum_{i=1}^n \expec^* \left[
\sup_{\vc{v}, \vc{w} \in \mathcal{V}_d(\eps,m)}
\abs{ Z_{n,i}(\vc{v}) - Z_{n,i}(\vc{w}) }^2
\right] \le c \eps^2
\]
for some $c > 0$ not depending on $\eps$ and such that $\int_0^\delta \sqrt{\log N_{d,\eps}} \, \diff \eps$ is finite for some $\delta > 0$. Uniting the coverings over $d = 1, \ldots, D$ yields a covering $\mathcal{V} = \bigcup_{d=1}^D \bigcup_{m=1}^{N_{d,\eps}} \mathcal{V}_d(\eps, m)$ by $N_{\eps} = \sum_{d=1}^D N_{d,\eps}$ sets.

Consider the covering sets $T_{\eps,j}$ for $j = 1, \ldots, N_{T,\eps}$ in Condition~\ref{cond:cover:T}. Further, let $[0, M] = \bigcup_{r=1}^{\lceil{M/\eps^2}\rceil} [l_{\eps,r}, u_{\eps,r}]$ be a covering of $[0, M]$ by intervals with lengths $u_{\eps,r} - l_{\eps,r}$ at most $\eps^2$. For $m \in \{1, \ldots, N_{T,\eps}\}^d \times \{1, \ldots, \lceil{M/\eps^2}\rceil\}^d = \mathcal{M}_\eps^d$, put
\begin{equation}
\label{eq:Vdem}
	\mathcal{V}_d(\eps,m)
	= \prod_{j=1}^d
	T_{\eps,{m_j}} \times
	\prod_{j=1}^d [l_{\eps,{m_{j+d}}}, u_{\eps,{m_{j+d}}}].
\end{equation}
For fixed $\eps > 0$, we have $\mathcal{V}_d = \bigcup_{m \in \mathcal{M}_\eps^d} \mathcal{V}_d(\eps,m)$.
\smallskip

	\emph{Step~4: Checking Eq.~\eqref{eq:entropy}.} ---
The number of covering sets is $N_\eps = \sum_{d=1}^D N_{d,\eps}$ where $N_{d,\eps} = \abs{\mathcal{M}_\eps^d} = (N_{T,\eps} \lceil{M/\eps^2}\rceil)^d$. Since $\sqrt{a + b} \le \sqrt{a} + \sqrt{b}$ for nonnegative $a$ and $b$, we have
\begin{align*}
	\sqrt{\log N_\eps}
	&\le \sqrt{\log(D N_{D,\eps})} \\
	&\le \sqrt{\log D} + \sqrt{D} \left(\sqrt{\log N_{T,\eps}} + \sqrt{\log(2M)} + \sqrt{2 \log(1/\eps)} \right).
\end{align*}
The condition $\int_0^{1} \sqrt{\log N_{T,\eps}} \, \diff \eps < \infty$ implies the same inequality for $N_{T,\eps}$ replaced by $N_\eps$.
\smallskip

	\emph{Step~5: Checking Ineq.\ \eqref{eq:ceps2}.} ---
We check that each covering set $\mathcal{V}_d(\eps,m)$, for $d = 1, \ldots, D$ and $m \in \mathcal{M}_\eps^d$, satisfies Ineq.~\eqref{eq:ceps2}, with $c > 0$ to be determined.
\smallskip

	\emph{Step~5.1: Bounding the supremum.} ---
	For $i = 1, \ldots, n$ and $j = 1, \ldots, d$, write
	\begin{align*}
		A_{i,j} &= \inf_{t \in T_{\eps, {m_j}}} U_i(t), &
		a_j &= \tfrac{k}{n} u_{\eps, {m_{j+d}}}, \\
		B_{i,j} &= \sup_{t \in T_{\eps, {m_j}}} U_i(t), &
		b_j &= \tfrac{k}{n} l_{\eps, {m_{j+d}}}.
	\end{align*}
	Put $\lambda = \exp(c_1 \eps^2)$ with $c_1$ as in Condition~\ref{cond:cover:T}. In view of the definition of $\mathcal{V}_d(\eps,m)$ in Eq.~\eqref{eq:Vdem}, we have
	\begin{align*}
		\sup_{\vc{v} \in \mathcal{V}_d(\eps,m)} Z_{n,i}(\vc{v})
		&=
		\sup_{\vc{v} \in \mathcal{V}_d(\eps,m)} k^{-1/2}
		\1 \{ \forall j : U_i(t_j) \le \tfrac{k}{n} x_j \} \\
		&\le
		k^{-1/2} \1 \{ \forall j : A_{i,j} \le a_j \} \\
		&\le
		k^{-1/2} \1 \{ \forall j : A_{i,j} \le a_j, B_{i,j} \le \lambda a_j \}
		+
		k^{-1/2} \1 \{ \exists j : A_{i,j} \le a_j, B_{i,j} > \lambda a_j \}.
	\end{align*}
	Similarly, we have
	\begin{align*}
		\inf_{\vc{v} \in \mathcal{V}_d(\eps,m)} Z_{n,i}(\vc{v})
		&=
		\inf_{\vc{v} \in \mathcal{V}_d(\eps,m)} k^{-1/2}
		\1 \{ \forall j : U_i(t_j) \le \tfrac{k}{n} x_j \} \\
		&\ge
		k^{-1/2} \1 \{ \forall j : B_{i,j} \le b_j \} \\
		&\ge
		k^{-1/2} \1 \{ \forall j : A_{i,j} \le \lambda^{-1} b_j, B_{i,j} \le b_j \}.
	\end{align*}
	Since $\lambda^{-1} b_j \le b_j \le a_j \le \lambda a_j$, we have
	\begin{align*}
		0
		&\le
		\1 \{ \forall j : A_{i,j} \le a_j, B_{i,j} \le \lambda a_j \}
		-
		\1 \{ \forall j : A_{i,j} \le \lambda^{-1} b_j, B_{i,j} \le b_j \} \\
		&\le
		\1 \{ \exists j : \lambda^{-1} b_j < A_{i,j} \le B_{i,j} \le \lambda a_j \}.
	\end{align*}
	It follows that
	\begin{align*}
	\lefteqn{
	\sup_{\vc{v}, \vc{w} \in \mathcal{V}_d(\eps,m)}
	\abs{ Z_{n,i}(\vc{v}) - Z_{n,i}(\vc{w}) }^2
	} \\
	&\le
	\abs{
		\sup_{\vc{v} \in \mathcal{V}_d(\eps,m)} Z_{n,i}(\vc{v})
		-
		\inf_{\vc{v} \in \mathcal{V}_d(\eps,m)} Z_{n,i}(\vc{v})
	}^2 \\
	&\le k^{-1}
	\left(
		\1 \{ \exists j : \lambda^{-1} b_j < A_{i,j} \le B_{i,j} \le \lambda a_j \}	 +
	 \1 \{ \exists j : A_{i,j} \le a_j, B_{i,j} > \lambda a_j \}
	 \right).
	\end{align*}
	\smallskip
	
	\emph{Step~5.2: Bounding the expectation.} ---
	By the previous step, we have
	\begin{multline}
	\label{eq:AAB}
		\sum_{i=1}^n \expec^* \left[
			\sup_{\vc{v}, \vc{w} \in \mathcal{V}_d(\eps,m)}
			\abs{ Z_{n,i}(\vc{v}) - Z_{n,i}(\vc{w}) }^2
		\right] \\
		\le
		\frac{n}{k} \sum_{j=1}^d \left[ \prob \left\{ \lambda^{-1} b_j < A_{1,j} \le B_{1,j} \le \lambda a_j \right\}
		+
		\prob \left\{ A_{i,j} \le a_j, \, B_{i,j} > \lambda a_j \right\}
		\right],
	\end{multline}
	since the stochastic processes $U_i$ are identically distributed. We treat each of the two probabilities in turn.
	
	First, picking any $t \in T_{\eps,{m_j}}$, we have $A_{1,j} \le U_1(t) \le B_{1,j}$ and thus, since $U_1(t)$ is uniformly distributed on $(0, 1)$, we have
	\begin{align*}
		\prob \left\{ \lambda^{-1} b_j < A_{1,j} \le B_{1,j} \le \lambda a_j \right\}
		&\le
		\prob \left\{ \lambda^{-1} b_j < U_1(t) \le \lambda a_j \right\} \\
		&\le
		\lambda a_j - \lambda^{-1} b_j \\
		&\le
		\tfrac{k}{n} \left( (\lambda - \lambda^{-1}) M + \eps^2 \right),
	\end{align*}
	where we used that $u_{\eps,{m_{j+d}}} \in [0, M]$ and $l_{\eps,{m_{j+d}}} \ge u_{\eps,{m_{j+d}}} - \eps^2$.
	
	Second, by Condition~\ref{cond:cover:T}, we have, for $n$ sufficiently large such that $(k/n)M$ and thus $a_j = (k/n) u_{\eps, {m_{j+d}}}$ is sufficiently small (Condition~\ref{cond:kn}),
	\begin{align*}
		\prob \left\{ A_{i,j} \le a_j, \, B_{i,j} > \lambda a_j \right\}
		&=
		\prob \left\{ B_{i,j} > \lambda a_j \mid A_{i,j} \le a_j \right\}
		\prob \left\{ A_{i,j} \le a_j \right\} \\
		&\le
		c_2 \eps^2
		\prob \left\{ \inf_{t \in T} U_1(t) \le \tfrac{k}{n} M \right\}.
	\end{align*}
	In view of Lemma~\ref{lem:tail}, the probability on the right-hand side is bounded by $c_3 (k/n) M$ for some $c_3 > 0$; the condition in Eq.~\eqref{eq:lem:tail} in that lemma is easily fulfilled thanks to Condition~\ref{cond:cover:T}.
	
	Putting things together, we obtain that the upper bound in Eq.~\eqref{eq:AAB} is, for sufficiently large $n$, dominated by
	\begin{align*}
		\tfrac{n}{k} D \left[
			\tfrac{k}{n} \left\{ (\lambda - \lambda^{-1}) M + \eps^2 \right\}
			+
			c_2 \eps^2 \cdot c_3 \tfrac{k}{n} M
		\right].
	\end{align*}
	Now $\lambda = \exp(c_1 \eps^2)$, so that, for $\eps > 0$ sufficiently close to zero, we have $0 \le \lambda - \lambda^{-1} \le 3 c_1 \eps^2$. We find the upper bound in Eq.~\eqref{eq:ceps2} with $c = D (3 c_1 M + 1 + c_2 c_3 M)$.
	\smallskip
	
	\emph{Step~6: Weak convergence and limit process.} ---
	By Steps~2 to~5, the criteria of Theorem~\ref{thm:tight} are met, so that the sequence of processes $W_n$ is asymptotically tight in $\ell^\infty(\mathcal{V})$. Moreover, the finite-dimensional distributions converge in view of Step~1. By \citet[Theorem~1.5.4]{vandervaart+w:1996}, $W_n$ converges weakly in $\ell^\infty(\mathcal{V})$ to a tight version of the process $W$, denoted by $W$ as well. This process $W$ is Gaussian, and by Example~1.5.10 in \citet{vandervaart+w:1996}, its tightness implies that the semimetric space $(\mathcal{V}, \rho_2)$ is totally bounded and that almost all sample paths $\vc{v} \mapsto W(\vc{v})$ are uniformly $\rho_2$-continuous, where $\rho_2$ is the standard deviation semimetric in Eq.~\eqref{eq:rho2}.
\end{proof}

\begin{proof}[Proof of Corollary~\ref{cor:main}]
Let $(s, t) \in T^2$ and $\lambda > 1$. We have
	\begin{align*}
		u^{-1} \prob[ U_1(s) \le u < \lambda u < U_1(t) ]
		&=
		u^{-1} \prob[ U_1(s) \le u ]
		- u^{-1} \prob[ U_1(s) \le u, U_1(t) \le \lambda u] \\
		&\to
		1 - R_{s,t}(1, \lambda), \qquad u \downarrow 0.
	\end{align*}
	Suppose in addition that $\rho(s, t) \le \delta$ for some $\delta > 0$.
  Put $B = \{ r \in T : \rho(r,s) \le \delta \}$. Then
	\begin{align*}
	\lefteqn{
		u^{-1} \prob[ U_1(s) \le u < \lambda u < U_1(t) ]
	} \\
		&\le
		u^{-1} \prob\left[
			\inf_{r \in B} U_1(r) \le u < \lambda u < \sup_{r \in B} U_1(r)
		\right] \\
		&\le
		u^{-1} \prob\left[ \inf_{r \in T} U_1(r) \le u \right]
		\prob\left[
			\sup_{r \in B} U_1(r) > \lambda u
			\, \left\lvert \,
			\inf_{r \in B} U_1(t) \le u
			\right.
		\right]
	\end{align*}
	Recall Condition~\ref{cond:cover:balls}. Let $\eps \in (0, \eps_0]$ and let $\delta(\eps)$ be as in that condition. Put $\lambda = \exp(c_1 \eps^2)$. By Lemma~\ref{lem:tail}, we have $S = \sup_{u > 0} u^{-1} \prob[ \inf_{r \in T} U_1(r) \le u ] < \infty$. As a consequence, we find
	\[
		u^{-1} \prob[ U_1(s) \le u < u \exp(c_1 \eps^2) < U_1(t) ]
		\le
		S c_2 \eps^2
	\]
	for all $u \in (0, u_0]$ and thus
	\[
		1 - R_{s,t}(1, \exp(c_1 \eps^2)) \le S c_2 \eps^2.
	\]
	By the $1$-Lipschitz property of tail copulas, we find
	\begin{align*}
		1 - R_{s,t}(1, 1)
		&=
		[1 - R_{s,t}(1, \exp(c_1 \eps^2))]
		+
		[R_{s,t}(1, \exp(c_1 \eps^2)) - R_{s,t}(1, 1)] \\
		&\le
		S c_2 \eps^2 + \exp(c_1 \eps^2) - 1.
	\end{align*}
	For a given $\eta > 0$, we can find $\eps = \eps(\eta)$ such that the previous upper bound is less than $\eta$. The choice $\delta = \delta(\eps(\eta))$ then fulfills the requirement.
	
	The last statement follows from the property just proved in combination with the increment bound in Lemma~\ref{lem:R:incr}.
\end{proof}

%

\section{Proofs for the results of Section~\ref{sec:emptailcop}}
\label{sec:proofs:emp}

We now prepare for the proof of Theorem~\ref{thm:main:2}. Recall $U_i(t) = 1 - F_t(\xi_i(t))$ for $(i,t) \in \{1,\ldots,n\} \times T$ and let $U_{1:n}(t) \le \ldots \le U_{n:n}(t)$ denote the  order statistics of $U_1(t), \ldots, U_n(t)$; set $U_{0:n}(t)=0$. Put
\begin{align*}
	K_{n,t}(x) &= \tfrac{n}{k} \, U_{\ceil{kx}:n}(t), \\
	\vc{K}_{n,\vc{t}}(\vc{x}) &= (K_{n,t_1}(x_1), \ldots, K_{n,t_d}(x_d)),
\end{align*}
for $(t,x) \in T \times [0, n/k]$ and $(\vc{t},\vc{x}) \in T^d \times [0, n/k]^d$. We will approximate the penultimate empirical tail copula process $\hat{W}_n$ by the process
\begin{align}
	\label{eq:decomp}
	\check{W}_n(\vc{v})
	&=
	\sqrt{k} \left\{
		G_{n,\vc{t}}\left(\vc{K}_{n,\vc{t}}(\vc{x})\right)
		-
		E_{n,\vc{t}}(\vc{x})
	\right\} \\
	\nonumber
	&=
	W_{n}(\vc{t}, \vc{K}_{n,\vc{t}}(\vc{x}))
	+
	\sqrt{k} \left\{
	E_{n,\vc{t}}\left(\vc{K}_{n,\vc{t}}(\vc{x})\right)
	- E_{n,\vc{t}}(\vc{x})
	\right\}	
\end{align}
for $\vc{v} = (\vc{t},\vc{x}) \in \mathcal{V}$.

\begin{lemma}
	\label{lem:Wch}
	Under the conditions of Theorem~\ref{thm:main}, we have
	\[
		\sup_{\vc{v} \in \mathcal{V}}
		\left|
			\hat{W}_{n}(\vc{v})
			-
			\check{W}_n(\vc{v})
		\right|
		\stackrel{p}{\to} 0, \qquad n \to \infty.
	\]
\end{lemma}

\begin{proof}
	For $t \in T$, let $\xi_{1:n}(t) \le \ldots \le \xi_{n:n}(t)$ denote the ascending order statistics of $\xi_1(t), \ldots, \xi_n(t)$. For $i,m \in \{1,\ldots,n\}$, we have
	\begin{align*}
	\1 \{ U_i(t) < U_{m:n}(t) \}
	&\le \1 \{ \xi_i(t) > \xi_{n-m+1:n}(t) \}
	= \1 \{ \operatorname{rank}_{n,i}(t) > n-m+1 \} \\
	&\le \1 \{ \xi_i(t) \ge \xi_{n-m+1:n}(t) \} \\
	&\le \1 \{ U_i(t) \le U_{m:n}(t) \}.
	\end{align*}
	Let $\vc{v} = (\vc{t},\vc{x}) \in \mathcal{V}$ and consider the right-hand limit
	\[
		G_{n,\vc{t}}(\vc{x}+)
		= \lim_{h \downarrow 0} G_{n,\vc{t}}(\vc{x}+h\vc{1})
		=
		\frac{1}{k} \sum_{i=1}^n \1\{\forall j=1,\ldots,d: U_i(t_j) \le \tfrac{k}{n}x_j \},
	\]
	where $\vc{1}$ denotes a vector of ones.
	It follows that
	\begin{align*}
		G_{n,\vc{t}}\left(\vc{K}_{n,\vc{t}}(\vc{x})\right)
		&=
		\frac{1}{k} \sum_{i=1}^n \1 \left\{
			\forall j = 1, \ldots, d:
			U_i(t_j) < U_{\ceil{k x_j}:n}(t_j)
		\right\} \\
		&\le
		\frac{1}{k} \sum_{i=1}^n \1 \left\{
			\forall j = 1, \ldots, d:
			\operatorname{rank}_{n,i}(t_j) > n - \ceil{kx_j} + 1
		\right\}
		=
		\hat{R}_{n,\vc{t}}(\vc{x}) \\
		&\le
		\frac{1}{k} \sum_{i=1}^n \1 \left\{
			\forall j = 1, \ldots, d:
			U_i(t_j) \le U_{\ceil{k x_j}:n}(t_j)
		\right\}
		= G_{n,\vc{t}}\left(\vc{K}_{n,\vc{t}}(\vc{x})+\right).
	\end{align*}
	To see the identity on the third line involving $\hat{R}_{n,\vc{t}}(\vc{x})$, note that, for integer $q$ and real $y$, we have $q < y \iff q < \ceil{y}$; as a consequence, the ceiling function can be omitted and we get the definition in Eq.~\eqref{eq:Rntx}.
	On the one hand, we have the lower bound
	\begin{align*}
		\hat{W}_n(\vc{v})
		&= \sqrt{k} \{ \hat{R}_{n,\vc{t}}(\vc{x}) - E_{n,\vc{t}}(\vc{x}) \} \\
		&\ge \sqrt{k} \left\{
			G_{n,\vc{t}}\left(\vc{K}_{n,\vc{t}}(\vc{x})\right)
			- E_{n,\vc{t}}(\vc{x})
		\right\}
		= \check{W}_n(\vc{v}).
	\end{align*}
	On the other hand, since the map $\vc{x} \mapsto E_{n,\vc{t}}(\vc{x})$ is continuous, using the same notation as above for right-hand limits, we have the upper bound
	\begin{align*}
		\hat{W}_n(\vc{v})
		&\le \sqrt{k} \left\{
		G_{n,\vc{t}}\left(\vc{K}_{n,\vc{t}}(\vc{x})+\right)
		- E_{n,\vc{t}}(\vc{x})
		\right\} \\
		&=  W_{n}(\vc{t}, \vc{K}_{n,\vc{t}}(\vc{x})+)
		+ \sqrt{k} \left\{
		E_{n,\vc{t}}(\vc{K}_{n,\vc{t}}(\vc{x}))
		- E_{n,\vc{t}}(\vc{x})
		\right\} \\
		&= \check{W}_{n}(\vc{v})
		+ \eta_{n,\vc{t}}(\vc{x}),
	\end{align*}
	where the remainder term is
	\[
		\eta_{n,\vc{t}}(\vc{x})
		= W_{n}(\vc{t}, \vc{K}_{n,\vc{t}}(\vc{x})+)
		- W_{n}(\vc{t}, \vc{K}_{n,\vc{t}}(\vc{x}))
		\ge 0.
	\]
	As noted right after Theorem~\ref{thm:main}, the sequence $W_n$ is asymptotically uniformly $\rho_2$-equicontinuous in probability.
	By the definition~\eqref{eq:rho2} of $\rho_2$ and the uniform continuity of the map $\vc{x} \mapsto R_{\vc{t}}(\vc{x})$, a property that follows from the Lipschitz property of a tail copula, we find that
	$\sup_{\vc{v} \in \mathcal{V}} \left| \eta_{n,\vc{t}}(\vc{x}) \right| 	\stackrel{p}{\to} 0$ as $n \to \infty$.
\end{proof}

We need to know the asymptotic behavior of the empirical tail quantile function $K_{n,t}$ uniformly in $t \in T$. Since $K_{n,t}$ is a generalized inverse of the tail empirical distribution function $G_{n,t}$ in (\ref{eq:Gntx}), we can obtain this  through a Vervaat-type lemma as given in Lemma 4 of \citet{einmahl+g+s:2010}, in combination with  Theorem \ref{thm:main} which gives the asymptotic behavior of  $G_{n,t}(x)$, uniformly in $(t, x) \in T \times [0, M]$.  Note that the functions in there are defined on $[0,1]$ but for our purposes should be defined on $[0,M+1]$ so that $G_{n,t}$ can be extended on $\left(M,M+1\right]$ in such a way that  $G_{n,t}(M+1)=M+1$.  Thus we obtain the following corollary to Theorem~\ref{thm:main}.

\begin{corollary}
	\label{lem:Vervaat:Z}
	Under the conditions of Theorem~\ref{thm:main}, we have
	\[
	\sup_{(t, x) \in T \times [0, M]} \left\lvert
	\sqrt{k} \{ K_{n,t}(x) - x \} + W_{n}(t,x)
	\right\rvert
	\stackrel{p}{\to} 0, \qquad n \to \infty.
	\]
\end{corollary}

%

\begin{proof} [Proof of Theorem~\ref{thm:main:2}]
Suppose that the asymptotic expansion of $\hat{W}_n$ in the convergence statement~\eqref{eq:thm:main:2:expansion} has been established. Define a mapping $\Phi$ from $\ell^\infty(\mathcal{V})$ to $\ell^\infty(\mathcal{V}')$ by
\[
	(\Phi(f))(\vc{v})
	=
	f(\vc{v}) - \sum_{j=1}^d \dot{R}_{\vc{t},j}(\vc{x}) \, f(t_j, x_j),
	\qquad \vc{v} = (\vc{t}, \vc{x}) \in \mathcal{V}'.
\]
The mapping is linear and bounded, since the partial derivatives of a tail copula are all between $0$ and $1$. By the continuous mapping theorem \citep[Theorem~1.3.6]{vandervaart+w:1996}, we have $\Phi(W_{n}) \dto \Phi(W) = \hat{W}$ as $n \to \infty$. Equation~\eqref{eq:thm:main:2:expansion} further says that the distance in $\ell^\infty(\mathcal{V}')$ between $\hat{W}_n$ and $\Phi(W_{n})$ converges to zero in probability as $n \to \infty$. By Slutsky's lemma (\citet[Example~1.4.7]{vandervaart+w:1996} or \citet[Theorem~18.10(iv)]{vandervaart:1998}), then also $\hat{W}_n \dto \hat{W}$ as $n \to \infty$.

The process $\hat{W}$ is Gaussian. It is also tight, since it arises as the image of a tight process by a continuous mapping, and the image of compact set by a continuous function is compact. By \citet[Example~1.5.10]{vandervaart+w:1996}, almost all trajectories of $\hat{W}$ are uniformly $\hat{\rho}_2$-continuous
on $\mathcal{V}'$. %

It remains to show the expansion in \eqref{eq:thm:main:2:expansion}. By Lemma~\ref{lem:Wch} and Slutsky's Lemma \citep[Example~1.4.7]{vandervaart+w:1996}, we can replace the penultimate tail empirical copula process $\hat{W}_n$ by its approximation $\check{W}_n$ defined in Eq.~\eqref{eq:decomp}. We analyze the two terms in the decomposition on the second line of \eqref{eq:decomp} separately.

\paragraph*{\bfseries The first term in \eqref{eq:decomp}}
We will show that
\begin{equation}
\label{eq:decomp:1}
	\sup_{(\vc{t},\vc{x}) \in \mathcal{V}} \left\lvert
		W_{n}(\vc{t}, \vc{K}_{n,\vc{t}}(\vc{x}))
		-
		W_{n}(\vc{t}, \vc{x})
	\right\rvert
	=
	o_p(1), \qquad n \to \infty.
\end{equation}
Recall the standard deviation metric $\rho_2$ in \eqref{eq:rho2}. Since $R_{\vc{t}, \vc{t}}(\vc{x}_1, \vc{x}_2) = R_{\vc{t}}(\vc{x}_1 \wedge \vc{x}_2)$, the minimum being coordinate-wise, we have
\begin{align*}
	\rho_{2}\left( (\vc{t}, \vc{x}_1), (\vc{t}, \vc{x}_2) \right)
	&=
	\left\{
		R_{\vc{t}}(\vc{x}_1)
		- 2 R_{\vc{t}}(\vc{x}_1 \wedge \vc{x}_2)
		+ R_{\vc{t}}(\vc{x}_2)
	\right\}^{1/2}.
\end{align*}
The map $\vc{x} \mapsto R_{\vc{t}}(\vc{x})$ is $1$-Lipschitz w.r.t.\ the $L_1$-norm $\abs{\point}_1$ on Euclidean space, whence
\begin{align}
	\label{eq:rho2:ineq}
	\rho_{2}\left( (\vc{t}, \vc{x}_1), (\vc{t}, \vc{x}_2) \right)
	&\le
	\left(
		\abs{\vc{x}_1 - \vc{x}_1 \wedge \vc{x}_2}_1
		+ \abs{\vc{x}_2 - \vc{x}_1 \wedge \vc{x}_2}_1
	\right)^{1/2} \\
\nonumber
	&\le
	2^{1/2} \abs{\vc{x}_1 - \vc{x}_2}_1^{1/2}.
\end{align}

Corollary~\ref{lem:Vervaat:Z} implies $\sup_{(t, x) \in T \times [0, M]} \left\lvert K_{n,t}(x) - x \right\rvert = o_p(1)$ as $n \to \infty$, and thus also
$
	\sup_{(\vc{t},\vc{x}) \in \mathcal{V}}
	\left\lvert \vc{K}_{n,\vc{t}}(\vc{x}) - \vc{x} \right\rvert_1
	=
	o_p(1)$ as $n \to \infty$.
By \eqref{eq:rho2:ineq}, we get
\[
	\sup_{(\vc{t},\vc{x}) \in \mathcal{V}}
	\rho_2\bigl(
	(\vc{t}, \vc{K}_{n,\vc{t}}(\vc{x})),
	(\vc{t}, \vc{x})
	\bigr)
	=
	o_p(1), \qquad n \to \infty.
\]
In view of Theorem~\ref{thm:main}, the processes $W_n$ are asymptotically uniformly $\rho_2$-equicontinuous in probability, see Example~1.5.10 in  \citet{vandervaart+w:1996}. Eq.~\eqref{eq:decomp:1} follows.

\paragraph*{\bfseries The second term in \eqref{eq:decomp}}
We aim to show that
\begin{equation}
\label{eq:decomp:2}
	\sup_{(\vc{t},\vc{x}) \in \mathcal{V}'}
	\left\lvert
		\sqrt{k} \left\{
			E_{n,\vc{t}}\left(\vc{K}_{n,\vc{t}}(\vc{x})\right)
			- E_{n,\vc{t}}(\vc{x})
		\right\}
		+
		\sum_{j=1}^d \dot{R}_{\vc{t},j}(\vc{x}) W_{n}(t_j, x_j)
	\right\rvert
	=
	o_p(1)
\end{equation}
as $n \to \infty$. To this end, we bound the absolute value in \eqref{eq:decomp:2} by the sum of three terms: 
\begin{align*}
	\Delta_{n,1}(\vc{t},\vc{x})
	&=
	\left\lvert
		\sqrt{k} \left\{
		E_{n,\vc{t}}\left(\vc{K}_{n,\vc{t}}(\vc{x})\right)
		- E_{n,\vc{t}}(\vc{x})
		\right\}
		-
		\sqrt{k} \left\{
			R_{\vc{t}}\left(\vc{K}_{n,\vc{t}}(\vc{x})\right)
			- R_{\vc{t}}(\vc{x})
		\right\}
	\right\rvert, \\
	\Delta_{n,2}(\vc{t},\vc{x})
	&=
	\sqrt{k} \left\lvert
		R_{\vc{t}}\left(\vc{K}_{n,\vc{t}}(\vc{x})\right)
		- R_{\vc{t}}\left(\vc{x} - \tfrac{1}{\sqrt{k}} \vc{W}_n(\vc{t},\vc{x})\right)
	\right\rvert,
\end{align*}
and
\begin{equation}
\label{eq:D3}
	\Delta_{n,3}(\vc{t},\vc{x})
	=
	\Biggl\lvert
		\sqrt{k} \left\{
			R_{\vc{t}}\left(\vc{x} - \tfrac{1}{\sqrt{k}} \vc{W}_n(\vc{t},\vc{x})\right)
			- R_{\vc{t}}(\vc{x})
		\right\} \\
		{} +
		\sum_{j=1}^d \dot{R}_{\vc{t},j}(\vc{x}) W_{n}(t_j, x_j)
	\Biggr\rvert,
\end{equation}
where $\vc{W}_n(\vc{t},\vc{x}) = (W_n(t_j,x_j) : j = 1, \ldots, d)$. For each of the three terms, we will show that the supremum over $(\vc{t},\vc{x})$ converges to zero in probability. Especially for the third term, the proof is challenging, as foreshadowed already in Remark~\ref{rem:challenges}.
\smallskip

\emph{The term $\Delta_{n,1}$.} ---
Recall Theorem~\ref{thm:main} and Corollary~\ref{lem:Vervaat:Z}. For $\eps > 0$, we can find $K > 0$ such that for sufficiently large $n$ and with probability at least $1-\eps$, we have $\sup_{t \in T, x \in [0, M]} \sqrt{k} \abs{K_{n,t}(x)-x} \le K$. On this event, the supremum of $\Delta_{n,1}(\vc{t},\vc{x})$ over all $(\vc{t},\vc{x}) \in \mathcal{V}_d$ is bounded by the supremum in Condition~\ref{cond:bias} and is therefore bounded by $\eps > 0$ for sufficiently large $n$. Since $\eps > 0$ was arbitrary, we find that the supremum of $\Delta_{n,1}(\vc{t},\vc{x})$ over $(\vc{t},\vc{x}) \in \mathcal{V}$ converges to zero in probability.
\smallskip

\emph{The term $\Delta_{n,2}$.} ---
The function $R_{\vc{t}}$ is 1-Lipschitz with respect to the $L_1$-norm, so $\Delta_{n,2}(\vc{t},\vc{x})$ is bounded by $\sum_{j=1}^d \lvert\sqrt{k}\{K_{n,t_j}(x_j) - x_j\} + W_{n}(t_j,x_j)\rvert$. The supremum of $\Delta_{n,2}(\vc{t},\vc{x})$ over all $(\vc{t},\vc{x}) \in \mathcal{V}$ is thus bounded by $D$ times the supremum in Corollary~\ref{lem:Vervaat:Z}, and therefore converges to zero in probability.
\smallskip

\emph{The term $\Delta_{n,3}$.} ---
In what follows, fix $d = 1, 2, \ldots,D .$ We consider the supremum of $\Delta_{n,3}(\vc{t},\vc{x})$ over $(\vc{t},\vc{x}) \in \mathcal{V}_d'$.
Fix $\eps > 0$ and let $\delta = \delta(\eps) > 0$ be chosen later on as a function of $\eps$ in such a way that a number of criteria are met. We bound the supremum as the maximum of the suprema over the $d$ sets $T^d \times \{\vc{x} \in [0, M]^d: x_j \le \delta \}$ for $j = 1, \ldots, d$ on the one hand (Case~I) and over $(T^d)' \times [\delta, M]^d$ (Case~II) on the other hand, where $(T^d)'$ comprises the $d$-tuples with all coordinates different and in $T$. We show that, for some constants $c_1, c_2 > 0$ not depending on $\eps$ or $n$, for all sufficiently large $n$, each of these suprema is less than $c_1 \eps$ with probability at least $1 - c_2 \eps$.
\smallskip

\underline{Case~I.} --
We first consider the supremum of $\Delta_{n,3}(\vc{t},\vc{x})$ over $\vc{t} \in T^d$ and $\vc{x} \in [0, M]^d$ with $x_j \le \delta$, for $j \in \{1, \ldots, d\}$ fixed and for $\delta = \delta(\eps) > 0$ to be chosen in function of a fixed $\eps > 0$.  For all $\eta > 0$ and for all $(t, x) \in T \times [0, \eta]$, we have
\[
	\rho_2((t, x), (t, 0)) = \sqrt{x} \le \sqrt{\eta},
\]
whereas $\rho_2((s,0), (t, 0)) = 0$ for $(s, t) \in T^2$. Furthermore, $W_n(t, 0) = 0$  for all $t \in T$. \modiv{Recall that the processes $W_{n}$ are asymptotically uniformly $\rho_2$-equicontinuous in probability.
We obtain that, for sufficiently small $\delta$, setting $\eta = \sqrt{\delta}$, we have $\rho_2((t, x), (t, 0)) \le \delta^{1/4}$ for all $(t, x) \in T \times [0, \sqrt{\delta}]$, so that} the supremum of $\abs{W_n(t, x)}$ over $(t, x) \in T \times [0, \sqrt{\delta}]$ is bounded by $\eps$ with probability at least $1-\eps$, for large $n$. On this event and for $\vc{t} \in T^d$ and $\vc{x} \in [0, M]^d$ such that $x_j \le \delta$, \modiv{the $1$-Lipschitz property of $R_{\vc{t}}$ implies that} the term $\Delta_{n,3}(\vc{t},\vc{x})$ is bounded by
\begin{equation}
\label{eq:lotje}
	\left\lvert
	\sqrt{k} \left\{
		R_{\vc{t}}\left( \vc{x} - \tfrac{1}{\sqrt{k}} \widetilde{\vc{W}}_{n,j}(\vc{t},\vc{x})\right) - R_{\vc{t}}(\vc{x})
	\right\}
	\right\rvert \\
	+
	\sum_{i \in \{1,\ldots,d\} \setminus \{j\}}
	\dot{R}_{\vc{t},i}(\vc{x}) \abs{W_n(t_i,x_i)}
	+
	2\eps,
\end{equation}
where the random $d$-vector $\widetilde{\vc{W}}_{n,j}(\vc{t},\vc{x})$ has the same $d$ coordinates $W_n(t_i, x_i)$ as $\vc{W}_{n}(\vc{t},\vc{x})$, except for the $j$th one, which is zero.

For the term $\dot{R}_{\vc{t},i}(\vc{x}) \abs{W_n(t_i,x_i)}$, we consider again two cases.
\begin{itemize}
	\item
	If, on the one hand, $x_i \le \sqrt{\delta}$, then $\abs{W_n(t_i,x_i)} \le \eps$ on the event already considered, while $0 \le \dot{R}_{\vc{t},i}(\vc{x}) \le 1$ by the $1$-Lipschitz property of $R_{\vc{t}}$.
	\item
	If, on the other hand, $x_i > \sqrt{\delta}$, then, by inequalities~\eqref{eq:xdotRbound} and~\eqref{eq:Rdiffbound} in Lemma~\ref{lem:cute},
	we have
	\[
		\dot{R}_{\vc{t},i}(\vc{x})
		\le R_{\vc{t}}(\vc{x}) / x_i
		\le x_j/x_i
		\le \delta/\sqrt{\delta}
		= \sqrt{\delta}.
	\]
	Moreover, given $\eps$, we can find $K = K(\eps) > 0$ such that the supremum of $\abs{W_n(t,x)}$ over $(t, x) \in T \times [0, M]$ is bounded by $K$ with probability at least $1-\eps$.
\end{itemize}
Hence, \modiv{with probability at least $1-2\eps$,} we have $0 \le \dot{R}_{\vc{t},i}(\vc{x}) \abs{W_n(t_i,x_i)} \le \max(\eps, K\sqrt{\delta})$ for all $\vc{t} \in T^d$ and all $\vc{x} \in [0, M]^d$ with $x_j \le \delta$.

For the first term in \eqref{eq:lotje}, a similar argument works on the same event (which has probability $1-2\eps$). For those indices $i \in \{1,\ldots,d\} \setminus \{j\}$ for which $x_i \le \sqrt{\delta}$, we can replace $W_{n}(t_i, x_i)$ by $0$ at the cost of an additional error of at most $\eps$. This amounts to an error of at most $(d-1)\eps$. For the remaining indices $i$, we write the increment as a telescoping sum of increments involving changes in one coordinate at a time. For each increment in the sum, we apply \eqref{eq:Rdiffbound}, with upper bound $x_j \le \delta$ and we proceed as in the second bullet point above.  This yield an error of at most $(d-1) K\sqrt{\delta}$.

\modiv{Given $\eps > 0$ and $j \in \{1,\ldots,d\}$, we have thus constructed an event with probability at least $1 - 2\eps$ such that the supremum of $\Delta_{n,\vc{t}}(\vc{x})$ over $\vc{t} \in T^d$ and $\vc{x} \in [0, M]^d$ with $x_j \le \delta$ is bounded by a constant multiple (the constant depending on $d$ only) of $\eps + K \sqrt{\delta}$, where $K$ depends on $\eps$ and where $\delta$ remains to be chosen. By choosing $\delta$ small enough, we can ensure that $K \sqrt{\delta} \le \eps$. We get that the supremum is less than a constant multiple of $\eps$ on an event with probability at least $1-2\eps$.}
\smallskip

\modiv{\underline{Case~II.} --} \label{page:caseII}
Finally we consider the supremum of $\Delta_{n,3}(\vc{t},\vc{x})$ in \eqref{eq:D3} over $(\vc{t}, \vc{x}) \in (T^d)' \times [\delta, M]^d$.
We write the supremum as the maximum over all partitions $\Pi$ of $\{1, \ldots, d\}$ of the suprema \modiv{over those $(\vc{t}, \vc{x})$ that satisfy the following property related to a given partition $\Pi$}: for indices $i,j \in \{1,\ldots,d\}$ within the same component of the partition, we have $\rho_2((t_i,x_i),(t_j,x_j)) \le \delta$ while for indices $i$ and $j$ in different components, we have $\rho_2((t_i,x_i),(t_j,x_j)) > \delta / (d-1)$. In this way, every point $(\vc{t}, \vc{x})$ is taken into account at least once in view of Lemma~\ref{lem:partition}.
The number of partitions being  finite, we can fix a single partition $\Pi = \{D_\alpha : \alpha \in A\}$ of $\{1, \ldots, d\}$ \modiv{and study the associated supremum}.
Consider the component $D_{\alpha(1)}$ of the partition that contains the index $1$, let $d_1 = |D_{\alpha(1)}| \in \{1, \ldots, d\}$ be the number of elements of this component, and assume for convenience of notation that this component is equal to $D_{\alpha(1)} = \{1, \ldots, d_1\}$.

\modiv{Let $(\vc{t},\vc{x})$ be a point satisfying the property associated to the given partition as described above.} The difference $\sqrt{k} \{
		R_{\vc{t}} (\vc{x} - \tfrac{1}{\sqrt{k}} \vc{W}_{n}(\vc{t},\vc{x}))
		-
		R_{\vc{t}} (\vc{x})
	\}$  can be written as a telescoping sum  with as many terms as the number $|A|$ of components $D_\alpha$ in the partition $\Pi$: we replace $\vc{W}_{n}(\vc{t},\vc{x})$ by the zero vector, one component $D_\alpha$ at a time.  We will only consider the first term of the telescoping sum, \modiv{i.e., the one in which the first $d_1$ elements of $\vc{W}_{n}(\vc{t},\vc{x})$ are replaced by zeroes}; the other terms can be treated similarly.
	
Let the random $d$-vector $\overline{\vc{W}}_n(\vc{t},\vc{x})$ have the last $d-d_1$ coordinates equal to those of $\vc{W}_{n}(\vc{t},\vc{x})$, whereas its first $d_1$ coordinates are all equal and are the average of the first $d_1$ coordinates of $\vc{W}_{n}(\vc{t},\vc{x})$. This average is denoted below by $\overline w$.
Similarly,  $\widetilde{\vc{W}}_n(\vc{t},\vc{x})$ has the last $d-d_1$ coordinates equal to those of $\vc{W}_{n}(\vc{t},\vc{x})$, whereas its first $d_1$ coordinates are all equal to zero.
We will consider the first term in the aforementioned telescoping sum:  $ \sqrt{k} \left\{
	R_{\vc{t}} \left( \vc{x} - \tfrac{1}{\sqrt{k}} \vc{W}_{n}(\vc{t},\vc{x}) \right)
	-
	R_{\vc{t}}  \left( \vc{x} - \tfrac{1}{\sqrt{k}} \widetilde{\vc{W}}_{n}(\vc{t},\vc{x}) \right)
	\right\}$.

Applying Lemma~\ref{lem:diagdiff} we obtain
\begin{equation}
\label{eq:bella}
	\left|\sqrt{k} \left\{
	R_{\vc{t}} \left( \vc{x} - \tfrac{1}{\sqrt{k}} \vc{W}_{n}(\vc{t},\vc{x}) \right)
	-
	R_{\vc{t}}  \left( \vc{x} - \tfrac{1}{\sqrt{k}} \overline{\vc{W}}_{n}(\vc{t},\vc{x}) \right)
	\right\} \right|
	\le
	 \sum_{j=1}^{d_1}\left|W_n(t_j,x_j)-\overline w\right|.
\end{equation}

\modiv{For the second term on the right-hand side in~\eqref{eq:bella}, we invoke again the asymptotic uniform $\rho_2$-equicontinuity in probability of the processes $W_n$. Since all points $(t_j,x_j)$ for $j \in \{1,\ldots,d_1\}$ are within $\rho_2$-distance $\delta$, the second term is uniformly bounded by $d_1\varepsilon$ with probability at least $1-\eps$, for $\delta$ small enough and for sufficiently large $n$.}

By the mean-value theorem, 
\begin{equation}
\label{kh}
 	\sqrt{k} \left\{
	R_{\vc{t}} \left( \vc{x} - \tfrac{1}{\sqrt{k}} \overline{\vc{W}}_{n}(\vc{t},\vc{x}) \right)
	-
	R_{\vc{t}}  \left( \vc{x} - \tfrac{1}{\sqrt{k}} \widetilde{\vc{W}}_{n}(\vc{t},\vc{x}) \right)
	\right\}
	= -\overline w\sum_{j=1}^{d_1}\dot{R}_{\vc{t},j}( \vc{\Theta}_n ),
\end{equation}
where $\vc{\Theta}_n = \vc{\Theta}_n(\vc{t},\vc{x})$ is a random point on the line segment connecting $\vc{x} - \tfrac{1}{\sqrt{k}} \widetilde{\vc{W}}_{n}(\vc{t},\vc{x})$ and $ \vc{x} - \tfrac{1}{\sqrt{k}}\overline{\vc{W}}_n(\vc{t},\vc{x})$. Clearly the first $d_1$ coordinates of $\vc{\Theta}_n-\vc{x} $ are all equal to $\overline h = \overline{h}_n(\vc{t},\vc{x})$, say. We would like to show that $\sum_{j=1}^{d_1}\dot{R}_{\vc{t},j}( \vc{\Theta}_n )$ is uniformly close to $\sum_{j=1}^{d_1}\dot{R}_{\vc{t},j}( \vc{x})$, with high probability.

Define the random $d$-vector $\vc{y}_n = \vc{y}_n(\vc{t},\vc{x})$ to have the last $d-d_1$ coordinates equal to those of $\vc{x}$ and the first $d_1$ coordinates   equal to those of $\vc{\Theta}_n$. For every $j \in \{1, \ldots, d_1\} $ we have from $d-d_1$ applications of~\eqref{d2} and Lemma~\ref{lem:dotR:unifcont} that $|\dot{R}_{\vc{t},j}( \vc{\Theta}_n )-\dot{R}_{\vc{t},j}( \vc{y}_n )|\le \varepsilon$ \modiv{uniformly over the considered $(\vc{t},\vc{x})$}, with probability at least $1-\varepsilon$, provided $\delta$ is small enough and $n$ is large.

Now Eq.~(\ref{eq:cuteid}) yields
\begin{align*}
	\sum_{j=1}^{d_1}\dot{R}_{\vc{t},j}( \vc{x} )
	&=\frac{1}{x_1}\sum_{j=1}^{d_1}x_j\dot{R}_{\vc{t},j}( \vc{x} )+\frac{1}{x_1}\sum_{j=1}^{d_1}(x_1-x_j)\dot{R}_{\vc{t},j}( \vc{x} )\\
	&=-\frac{1}{x_1}\sum_{j=d_1+1}^{d}x_j\dot{R}_{\vc{t},j}( \vc{x} )+ \frac{1}{x_1}  R_{\vc{t}}( \vc{x} )+\frac{1}{x_1}\sum_{j=1}^{d_1}(x_1-x_j)\dot{R}_{\vc{t},j}( \vc{x} ),
\end{align*}
and we can rewrite  $\sum_{j=1}^{d_1}\dot{R}_{\vc{t},j}( \vc{y}_n )$ similarly. Hence by \eqref{eq:rho2:sum} 
it follows after some algebra that
\begin{multline*}
	\left|
		\sum_{j=1}^{d_1}\dot{R}_{\vc{t},j}( \vc{y}_n )
		- \sum_{j=1}^{d_1}\dot{R}_{\vc{t},j}( \vc{x} )
	\right|
	\le
	\frac{M}{\delta} \sum_{j=d_1+1}^{d} \left|
		\dot{R}_{\vc{t},j}( \vc{y}_n )
		- \dot{R}_{\vc{t},j}( \vc{x} )
	\right| \\
	{} +
	\frac{d_1|\overline h|}{\delta}
	+
	\frac{1}{\delta(\delta+\overline h)} dM|\overline h|
	+
	2d_1\frac{\delta^2}{\delta\wedge (\delta +\overline h)}.
\end{multline*}
Take $\delta$ such that $3d_1\delta\le\eps$. \modiv{Then since $\overline{h} = O_p(1/\sqrt{k})$ uniformly in $(\vc{t},\vc{x})$, we have, with probability at least $1-\varepsilon$, that the sum of the three terms on the second line is uniformly bounded by $\eps$ when $n$ is large.}

So it remains to consider $|\dot{R}_{\vc{t},j}( \vc{y}_n )-\dot{R}_{\vc{t},j}( \vc{x} )|$ for $j\in \{d_1+1, \ldots, d\}$, that is, the first $d_1$ coordinates of $\vc{y}_n$ have to be replaced consecutively by those of $\vc{x}$ by writing the difference as a telescoping sum. We will consider for convenience only \modiv{the first term in that sum}. Let $\vc{y}_n'$ have the same coordinates as $\vc{y}_n$, except for the first one, which is $x_1$.
By Ineq.~\eqref{d2} we have
\[
	|\dot{R}_{\vc{t},j}( \vc{y}_n )-\dot{R}_{\vc{t},j}( \vc{y}_n' )|
	\le
	|\dot{R}_{(t_1,t_j),2}(x_1+\overline h, x_j )
	-\dot{R}_{(t_1,t_j),2}(x_1,x_j )|
.
\]
By Lemma~\ref{lem:dotR:unifcont}, the latter difference is uniformly small since $\overline h$ tends to zero in probability, uniformly.

\modiv{We have thus shown that in~\eqref{kh}, with large probability and for large $n$, the sum $\sum_{j=1}^{d_1} \dot{R}_{\vc{t},j}(\vc{\Theta}_n)$ may be replaced by $\sum_{j=1}^{d_1} \dot{R}_{\vc{t},j}(\vc{x})$ at the cost of an error bounded by a constant multiple of $\eps$. It remains to replace $\overline w=\frac{1}{d_1} \sum_{l=1}^{d_1} W_n(t_l,x_l)$ in~\eqref{kh} by $W_n(t_j,x_j)$ for $j\in \{1, \ldots, d_1\}$, written inside the sum over $j = 1, \ldots, d_1$. The property that the resulting difference is small follows from the asymptotic $\rho_2$-equicontinuity in probability of $W_n$.}
\modiv{We have thereby found a suitable upper bound on the supremum of $\Delta_{n,3}(\vc{t},\vc{x})$ over points $(\vc{t},\vc{x}) \in (T^d)' \times [\delta,M]^d$ that satisfy the property associated to a fixed partition $\Pi$ of $\{1,\ldots,d\}$ as described in the opening paragraph of Case~II. This completes the analysis of the term $\Delta_{n,3}$ in Eq.~\eqref{eq:D3} and thus of the convergence statement~\eqref{eq:thm:main:2:expansion}. }
\end{proof}

\begin{proof}[Proof of Corollary~\ref{cor:main2}]
	The supremum in Condition~\ref{cond:bias} is bounded by twice the supremum in Condition~\ref{cond:bias2}, if we replace $M$ in the latter condition by $M+1$ and if $k$ is sufficiently large. Hence, Condition~\ref{cond:bias2} implies Condition~\ref{cond:bias}.
	
	The difference between $\hat{W}_n(\vc{v})$ and $\sqrt{k} \{\hat{R}_{n,\vc{t}}(\vc{x}) - R_{\vc{t}}(\vc{x})\}$ is bounded by the maximum over $d = 1, \ldots, D$ of the suprema in Condition~\ref{cond:bias2}. Under that condition, the two processes must thus share the same first-order asymptotic expansion and the same weak convergence properties in $\ell^\infty(\mathcal{V})$ and thus also in $\ell^\infty(\mathcal{V}')$.
\end{proof}


\begin{appendix}
	\section{Auxiliary results}
	\label{app:aux}
	
	%
	
	\begin{lemma}
		\label{lem:R:incr}
		\modiv{Let $R$ be as in Condition~\ref{cond:tailCopula}.}
		For $\vc{s}, \vc{t} \in T^d$ and $\vc{x}, \vc{y} \in [0, \infty)^d$, we have
		\[
		\abs{R_{\vc{s}}(\vc{x}) - R_{\vc{t}}(\vc{y})}
		\le
		\sum_{j=1}^d
		[\abs{x_j-y_j} + (x_j \wedge y_j) \{1 - R_{s_j,t_j}(1, 1)\}].
		\]
	\end{lemma}
	
	
	\begin{proof}
		Let $\bar{s}, s \in T$ and $\bar{\vc{t}} \in T^m$ for some $m$. Then
		\begin{align*}
		R_{\bar{s}, \bar{ \vc{t} } } (x, \bar{\vc{y}}) - R_{s,\bar{\vc{t}}}(x, \bar{\vc{y}} )
		&=
		R_{\bar{s},s, \bar{\vc{t}}}(x,\infty, \bar{\vc{y}}) - R_{\bar{s},s,\bar{\vc{t}}}(\infty, x, \bar{\vc{y}} ) \\
		&\le R_{\bar{s},s, \bar{\vc{t}}}(x,\infty, \bar{\vc{y}}) - R_{\bar{s},s,\bar{\vc{t}}}(x, x, \bar{\vc{y}} )\\
		&\le
		R_{\bar{s},s}(x,\infty) - R_{\bar{s},s}(\infty, x )  \\
		&=
		x - R_{\bar{s},s}(x, x) \\
		&=
		x \left( 1 - R_{\bar{s},s}(1, 1) \right).
		\end{align*}
		The same inequality holds with the roles of $\bar{s}$ and $s$ interchanged. We obtain
		\begin{equation}
		\label{eq:R:sst}
		\abs{R_{\bar{s}, \bar{ \vc{t} } }(x, \bar{\vc{y}}) - R_{s, \bar{\vc{t}}}(x, \bar{\vc{y}})}
		\le
		x \left( 1 - R_{\bar{s},s}(1, 1) \right).
		\end{equation}

		Consider the statement of the lemma. Let $z_j = x_j \wedge y_j$ for $j = 1, \ldots, d$. \modiv{Then $\abs{R_{\vc{s}}(\vc{x}) - R_{\vc{t}}(\vc{y})}$ is bounded by}
		\begin{equation}
		\label{eq:R:incr:aux}
		\abs{R_{\vc{s}}(\vc{x}) - R_{\vc{s}}(\vc{z})}
		+
		\abs{R_{\vc{s}}(\vc{z}) - R_{\vc{t}}(\vc{z})}
		+
		\abs{R_{\vc{t}}(\vc{z}) - R_{\vc{t}}(\vc{y})}.
		\end{equation}
		The sum of the first and third terms in \eqref{eq:R:incr:aux} is bounded by
		\[
		\sum_{j=1}^d (\abs{z_j - x_j} + \abs{z_j - y_j})
		=
		\sum_{j=1}^d \abs{x_j - y_j}.
		\]
		The middle term in \eqref{eq:R:incr:aux} can be written as a telescoping sum of $d$ terms, in which the coordinate $s_j$ is replaced by $t_j$ one coordinate at a time. Apply the triangle inequality to bound the absolute value of the sum by the sum of the absolute values. The $j$th term in the resulting sum is bounded by $z_j (1 - R_{s_j,t_j}(1, 1))$ in view of \eqref{eq:R:sst}.
	\end{proof}
	
	Note that \eqref{eq:R:sst} shows that $R_{\vc{t}}, \vc{t} \in T^d$, is continuous in each $t_j$ with respect to the $\rho_2$-semimetric on $T$.
	By Theorem~25.7 in  \citet{rockafellar:1970}, the partial derivatives of $R_{\vc{t}}(\vc{x})$ with respect to $x_j$, for $j \in \{1,\ldots, d\}$, are then continuous in $\vc{t}$ as well.
	
	\begin{lemma}
		\label{lem:partition}
		Let $x_1, \ldots, x_d$ be $d$ points in a semimetric space $(\mathcal{X}, \tilde \rho)$ and let $\delta > 0$. Then there exists a partition $\Pi$ of $\{1,\ldots,d\}$ such that for all indices $i, j$ that belong to the same component of the partition we have $\tilde \rho(x_i,x_j) \le \delta$ while for all indices $i, j$ that belong to different components  of the partition we have $\tilde \rho(x_i, x_j) > \delta / (d-1)$.
	\end{lemma}
	
	\begin{proof}
		Construct a graph on $\{1, \ldots, d\}$ by connecting different indices $i$ and $j$ as soon as $\tilde \rho(x_i,x_j) \le \delta / (d-1)$. Let $\Pi$ be the partition formed by the connectivity components of the graph. Indices $i$ and $j$ that belong to different components are not connected by an edge and thus $\tilde \rho(x_i,x_j) > \delta / (d-1)$. Indices $i$ and $j$ that belong to the same component are linked up through a chain of at most $d-1$ edges. By the triangle inequality, $\tilde \rho(x_i,x_j) \le \delta$.
	\end{proof}
	
	\begin{lemma}
		\label{lem:cute}
		If a $d$-variate tail copula $R$ is differentiable on $(0, \infty)^d$, then, for all $\vc{x} \in [0, \infty)^d$, we have
		\begin{equation}
		\label{eq:cuteid}
		\sum_{j=1}^d x_j \dot{R}_j(\vc{x}) = R(\vc{x})
		\end{equation}
		\modiv{(both sides being zero if $x_j = 0$ for some $j = 1, \ldots, d$)}
		and hence for all $j\in \{1, \ldots, d\}$,
		\begin{equation}
		\label{eq:xdotRbound}
		x_j \dot{R}_j(\vc{x}) \le R(\vc{x}).
		\end{equation}
		In addition,
		for all $h \ge -x_1$ such that $h \neq 0$, we have
		\begin{equation}
		\label{eq:Rdiffbound}
		x_1 \frac{R(x_1+h,x_2,\ldots,x_d) - R(\vc{x})}{h}
		\le R(\vc{x})
		\le \min(x_1, \ldots, x_d).
		\end{equation}
	\end{lemma}
	
	\begin{proof}
		Evaluating the derivative of the function $\lambda \mapsto R(\lambda \vc{x}) = \lambda R(\vc{x})$ in $\lambda = 1$ yields the identity in (\ref{eq:cuteid}) on $(0, \infty)^d$.
		The identity extends to $\vc{x} \in [0, \infty)^d$ since for vectors $\vc{x}$ of which one or more coordinates are zero, both sides of the above equation are zero: if $x_j = 0$, then the definition of $\dot{R}_j(\vc{x})$ is immaterial, while for $i \ne j$, we have $\dot{R}_j(\vc{x}) = 0$.
		
		\modiv{Inequality \eqref{eq:xdotRbound} is a consequence of~\eqref{eq:cuteid} and the fact that $R$ is non-decreasing in each variable.}
		
		%
		%
		%
		%
		For the first inequality in \eqref{eq:Rdiffbound}: for $h = -x_1$, the inequality becomes an equality, and by concavity of $R$, the left-hand side is non-increasing in $h$. The second inequality in \eqref{eq:Rdiffbound} is immediate from the definition (and well-known).
	\end{proof}
	
	\modiv{Let $\rho_2$ be the standard deviation semimetric associated to a spatial tail copula $R$ as introduced in Eq.~\eqref{eq:rho2} in the main paper.}
	
	\begin{lemma}
		For $s,t \in T$ and $x,y \in [0, \infty)$, we have
		\begin{align}
		\label{eq:rho2:sum}
		[\rho_2((s,x),(t,y))]^2
		&= \abs{x - y} + 2 \left( (x \wedge y) - R_{s,t}(x, y) \right) \\
		\label{eq:rho2:sum:ineq}
		&\le \abs{x-y} + 2 (x \wedge y)(1 - R_{s,t}(1,1)).
		\end{align}
	\end{lemma}
	
	\begin{proof}
		Since $x - 2 (x \wedge y) + y 
		= \abs{x-y}$, we have
		\begin{align*}
		\rho_2((s, x), (t, y))^2
		&=
		\expec[ \{W(s,x) - W(t,y)\}^2 ] \\
		&=
		x - 2 R_{s,t}(x, y) + y \\
		&=
		x - 2 (x \wedge y) + y + 2 \left( (x \wedge y) - R_{s,t}(x, y) \right) \\
		&=
		\abs{x - y} + 2 \left( (x \wedge y) - R_{s,t}(x, y) \right),
		\end{align*}
		which is \eqref{eq:rho2:sum}. The inequality in \eqref{eq:rho2:sum:ineq} then follows from $R_{s,t}(x,y) \ge R_{s,t}(x \wedge y, x \wedge y) = (x \wedge y) R_{s,t}(1,1)$.
	\end{proof}
	
	\modiv{Recall the semimetric $\rho_{2,T}$ on $T$ induced by $\rho_2$ as defined in Eq.~\eqref{eq:rho2T} in the main paper.} Then \eqref{eq:rho2:sum:ineq} states that
	\[
	[\rho_2((s, x), (t, y))]^2
	\le
	\abs{x - y} + (x \wedge y) [\rho_{2,T}(s,t)]^2,
	\]
	splitting the semimetric on $T \times [0, \infty)$ into a component on $T$ and a component on $[0, \infty)$.
	
	\begin{lemma}
		\label{lem:rho2zero}
		For $(s_n,x_n), (s,x) \in T \times [0, \infty)$, we have, as $n \to \infty$,
		\[
		\rho_2((s_n,x_n),(s,x)) \to 0 \\
		\iff
		[x_n \to x \text{ and } x (R_{s_n,s}(1, 1) - 1) \to 0].
		\]
	\end{lemma}
	
	\begin{proof}
		By \eqref{eq:rho2:sum}, we have
		\[
		\rho_2((s_n,x_n),(s,x))^2
		=
		\abs{x_n-x} + 2 \left( (x_n \wedge x) - R_{s_n,s}(x_n, x) \right).
		\]
		If $x = 0$, then $\rho_2((s_n,x_n),(s,x))^2 = x_n$, which converges to zero if and only if $x_n \to 0 = x$, as required. Henceforth, suppose $x > 0$.
		
		On the one hand, if $x_n \to x$ and $R_{s_n,s}(1, 1) \to 1$, then by \eqref{eq:rho2:sum:ineq} we also have $\rho_2((s_n,x_n),(s,x)) \to 0$.
		
		On the other hand, if $\rho_2((s_n,x_n),(s,x)) \to 0$, then by \eqref{eq:rho2:sum}, both $x_n \to x$ and $(x_n \wedge x) - R_{s_n,s}(x_n, x) \to 0$. Since $x_n \vee x - x_n \wedge x \to 0$, also $x_n \vee x - R_{s_n,s}(x_n, x) \to 0$. But $R_{s_n,s}(x_n, x) \le R_{s_n,s}(x_n \vee x, x_n \vee x) = (x_n \vee x) R_{s_n,s}(1, 1)$, and thus also $(x_n \vee x)(1 - R_{s_n,s}(1, 1)) \to 0$.
	\end{proof}

	\begin{lemma}
		\label{lem:diagdiff}
		Let $R$ be a $d$-variate tail copula, where $d = d_1 + d_2$. For $(\vc{x},\vc{y}) \in [0,\infty)^{d_1} \times [0, \infty)^{d_2}$ and for $\vc{h} \in \reals^{d_1}$ such that $x_j + h_j \ge 0$ for all $j = 1, \ldots, d_1$, we have
		\[
		\abs{R(\vc{x}+\vc{h},\vc{y}) - R(\vc{x}+\bar{h}\vc{1},\vc{y})}
		\le
		\sum_{j=1}^{d_1} \abs{h_j - \bar{h}}
		\]
		with $\bar{h} = (h_1+\cdots+h_{d_1})/d_1$ and with $\vc{1} = (1, \ldots, 1)$ of dimension $d_1$.
	\end{lemma}
	
	\begin{proof}
		This follows immediately from the $1$-Lipschitz property of a tail copula with respect to the $L_1$-norm.
	\end{proof}

	\begin{lemma}
		Let $R$ be a $d$-variate ($d\ge 2)$ tail copula. Assume its partial derivatives $\dot{R}_j$ exist. Write $R(x,y)$ for $R(x,y, \vc{\infty})$. For $g>0$ and $(x, y, \vc{z}) \in [0, \infty) \times [0, \infty) \times [0, \infty)^{d-2}$, we have
		\begin{equation}\label{d2}
		0
		\le \dot{R}_1(x,y+g,\vc{z}) - \dot{R}_1(x,y,\vc{z})
		\le \dot{R}_1(x,y+g) - \dot{R}_1(x,y).
		\end{equation}
	\end{lemma}
	
	\begin{proof} Since $R$ is the distribution function of a nonnegative Borel measure on $[0, \infty)^d$, we have
		\begin{align*}
		0
		&\le
		\frac{R(x+h,y+g,\vc{z}) - R(x,y+g,\vc{z})}{h}
		-
		\frac{R(x+h,y,\vc{z}) - R(x,y,\vc{z})}{h}
		\\
		&\le
		\frac{R(x+h,y+g) - R(x,y+g)}{h}
		-
		\frac{R(x+h,y) - R(x,y)}{h}.
		\end{align*}
		Taking the limit as $h \downarrow 0$, we find (\ref{d2}) (for the right-hand partial derivatives).
	\end{proof}
	
	\begin{lemma}
		\label{lem:dotR:cont}
		Assume Condition~\ref{cond:Rdot}. Let $s_n,t_n,s,t \in T$ and $x_n,y_n,x,y \in [0, \infty)$. If $x > 0$ and $\rho_2((s,x),(t,y)) > 0$ and if $\rho_2((s_n,x_n),(s,x)) \to 0$ and $\rho_2((t_n,y_n),(t,y)) \to 0$, then $\dot{R}_{s_n,t_n;1}(x_n,y_n) \to \dot{R}_{s,t;1}(x,y)$ as $n \to \infty$.
	\end{lemma}
	
	\begin{proof}
		Since $x > 0$, the assumption $\rho_2((s_n,x_n),(s,x)) \to 0$ is equivalent to $x_n \to x$ and $R_{s_n,s}(1,1) \to 1$ by Lemma~\ref{lem:rho2zero}. We consider two cases: $y = 0$ and $y > 0$.
		
		First suppose $y = 0$, so that $\dot{R}_{s,t;1}(x,y) = 0$. Since $\rho_2((t_n,y_n),(t,y)) \to 0$, Lemma~\ref{lem:rho2zero} implies that $y_n \to 0$ as well. By \eqref{eq:xdotRbound}, we have $\dot{R}_{s_n,t_n;1}(x_n,y_n) \le x_n^{-1} R_{s_n,t_n}(x_n,y_n) \le x_n^{-1}y_n \to 0$ too, since $x_n \to x > 0$.
		
		Second suppose $y > 0$. By Lemma~\ref{lem:rho2zero}, the assumption $\rho_2((t_n,y_n),(t,y)) \to 0$ implies $y_n \to y$ and $R_{t_n,t}(1, 1) \to 1$. Lemma~\ref{lem:R:incr} ensures that the functions $f_n = R_{s_n,t_n}$ converge to the function $f = R_{s,t}$ uniformly on bounded subsets of $[0, \infty)^2$. Since $\rho_2((s,x),(t,y)) > 0$, we have $x \ne y$ or $R_{s,t}(1, 1) < 1$ (or both) in view of Eq.~\eqref{eq:rho2:sum}. In any case, the functions $f$ and $f_n$ (for sufficiently large $n$) are continuously differentiable in an open neigborhood $O$ of $(x, y)$: if $R_{s,t}(1, 1) < 1$ this is true by Condition~\ref{cond:Rdot}, while if $R_{s,t}(1, 1) = 1$, then this is true by the fact that $f = \min$ and $x \ne y$ and  for $f_n$ we use either Condition~\ref{cond:Rdot} or that $f_n = \min$. Let $N$ be a compact neighbourhood of $(x, y)$ contained in $O$. By \citet[Theorem~25.7]{rockafellar:1970}, the gradients of $f_n$ converge to the one of $f$ uniformly on $N$. Convergence of the partial derivatives as stated in the lemma follows.
	\end{proof}

	\begin{lemma}
		\label{lem:complete}
		If the semimetric space $(T, \rho_{2,T})$ is complete, then so is the semimetric space $(T \times [0, \infty), \rho_2)$
	\end{lemma}
	
	\begin{proof}
		Let $(s_n,x_n)_{n}$ be a $\rho_2$-Cauchy sequence in $T \times [0, \infty)$: for every $\eps > 0$, there exists an integer $n(\eps) \ge 1$ such that for all integers $k, \ell \ge n(\eps)$ we have $[\rho_2((s_k,x_k),(s_\ell,x_\ell))]^2 \le \eps$. We need to find $(s,x) \in T \times [0, \infty)$ with the property that $\rho_2((s_n,x_n),(s,x)) \to 0$ as $n \to \infty$.
		
		By \eqref{eq:rho2:sum}, we have $\abs{x_k-x_\ell} \le [\rho_2((s_k,x_k),(s_\ell,x_\ell))]^2 \le \eps$ whenever $k, \ell \ge n(\eps)$. Hence	$(x_n)_n$ is a Cauchy sequence in $[0, \infty)$. Since $[0, \infty)$ is complete, there exists $x \in [0, \infty)$ so that $x_n \to x$ as $n \to \infty$.
		
		Suppose first that $x = 0$. Then for every $s \in T$, we have $\rho_2((s_n,x_n),(s,0)) = x_n \to 0$ as $n \to \infty$, so that $(s_n,x_n)$ converges to $(s, 0)$. [Note that $\rho_2$ is a semimetric on $T \times [0, \infty)$ since $\rho((s,0),(t,0)) = 0$ for any $s, t \in T$, even if $\rho_{2,T}$ is a metric on $T$.]
		
		Suppose next that $x > 0$. Then by \eqref{eq:rho2:sum} we have $(x_k \wedge x_\ell) - R_{s_k,s_\ell}(x_k, x_\ell) \le \eps/2$ whenever $k, \ell \ge n(\eps)$. At the same time, we can increase $n(\eps)$ if necessary to ensure that also $\abs{x_k \wedge x_\ell - x} \le \eps/4$ and $\abs{R_{s_k,s_\ell}(x_k,x_\ell) - R_{s_k,s_\ell}(x,x)} \le \abs{x_k-x} + \abs{x_\ell-x} \le \eps/4$ for all integer $k,\ell \ge n(\eps)$. But then also $x(1 - R_{s_k,s_\ell}(1, 1)) = x - R_{s_k,s_\ell}(x,x) \le \eps$ for all $k, \ell \ge n(\eps)$. Hence, $(s_n)_n$ is a Cauchy sequence in $(T, \rho_{2,T})$. By assumption, we can find $s \in T$ such that $1 - R_{s_n,s}(1, 1) \to 1$ as $n \to \infty$. But then $(s_n, x_n)$ converges to $(s, x)$ with respect to $\rho_2$, as required.
	\end{proof}
	
	\begin{lemma}
		\label{lem:dotR:unifcont}
		Assume Condition~\ref{cond:Rdot} and assume that $(T, \rho_{2,T})$ is totally bounded. Let $0 < \delta < M < \infty$ and consider the set
		\[
		K = \{
		(s, x; t, y) \in (T \times [\delta, M]) \times (T \times [0, M]) :
		\rho_2((s, x), (t, y)) \ge \delta
		\}.
		\]
		For every $\eps > 0$, there exists $\eta > 0$ such that for all $(s_i,t_i;x_i,y_i) \in K$, $i \in \{1,2\}$, we have
		\begin{align*}
		&\rho_2((s_1,x_1),(s_2,x_2)) + \rho_2((t_1,y_1), (t_2,y_2))
		\le \eta \\
		\implies &
		\abs{ \dot{R}_{s_1,t_1;1}(x_1,y_1) - \dot{R}_{s_2,t_2;1}(x_2,y_2) }
		\le \eps.
		\end{align*}
	\end{lemma}
	
	\begin{proof}
		Consider the product space $\mathbb{D} = (T \times [0, \infty)) \times (T \times [0, \infty))$ equipped with the sum semimetric induced by the semimetric $\rho_2$ on $T \times [0, \infty)$. By Lemma~\ref{lem:complete}, the product space $\mathbb{D}$ is complete. The set $K$ in the statement is a closed subset of $\mathbb{D}$ and is thus complete as well. Moreover, $K$ is totally bounded, since $(T, \rho_{2,T})$ is totally bounded and since $\rho_2$ can be bounded as in \eqref{eq:rho2:sum:ineq}. As a consequence, $K$ is compact with respect to the sum semimetric induced by $\rho_2$.
		
		By Lemma~\ref{lem:dotR:cont}, the function $f$ with domain $\mathbb{D}_1 = \{(s,x;t,y) \in \mathbb{D} : \rho_2((s,x),(t,y)) > 0, x > 0 \}$ and defined by $f(s,x;t,y) = \dot{R}_{s,t;1}(x,y)$ is continuous. Since $K$ is a compact subset of $\mathbb{D}_1$, the function $f$ is uniformly continuous on $K$. But this is exactly the statement of the lemma.
	\end{proof}

	\section{Testing stationarity: further details}
	\label{app:testStat}
	
	We provide additional details to the results and calculations concerning the stationarity test in Section~\ref{sec:testStat} in the paper. Notation not defined here is as in that section.
	
	\subsection{The limit distribution of the test statistic (fixed $N$)}
	
	The population equivalent of the statistic $\hat{I}_n^{(N)}$ in Eq.~\eqref{eq:InN} is
	\[
	I^{(N)}(r_0/N)
	= \frac{1}{2\Delta} \sum_{r : 1 \le |r - r_0| \le \Delta}
	R_{r/N, r_0/N}(1, 1),
	\qquad r_0 \in \{\Delta, \ldots, N-\Delta\}.
	\]
	Under the null hypothesis of stationarity, $I^{(N)}(r_0/N)$ is constant in $r_0$. As a consequence, the test statistic in Eq.~\eqref{eq:DnN} is equal to
	\[
	D_n^{(N)}
	=
	\max_{r_0 = \Delta, \ldots, N-\Delta} V_n^{(N)}(r_0/N)
	-
	\min_{r_0 = \Delta, \ldots, N-\Delta} V_n^{(N)}(r_0/N)
	\]
	where
	\begin{align}
	\nonumber
	V_n^{(N)}(r_0/N)
	&=
	\sqrt{k} \left\{ \hat{I}^{(N)}(r_0/N) - I^{(N)}(r_0/N) \right\} \\
	&=
	\frac{1}{2\Delta} \sum_{r : 1 \le |r-r_0| \le \Delta}
	\sqrt{k}
	\left\{ \hat{R}_{n;r/N,r_0/N}(1, 1) - R_{r/N,r_0/N}(1, 1) \right\}
	\label{eq:VnN}
	\end{align}
	for $r_0 \in \{\Delta, \ldots, N-\Delta\}$. By the continuous mapping theorem, Corollary~\ref{cor:main2} implies that $V_n^{(N)} \dto V^{(N)}$ and thus $D_n^{(N)} \dto D^{(N)}$ as $n \to \infty$, with $V^{(N)}$ and $D^{(N)}$ as in Eq.~\eqref{eq:VN} and~\eqref{eq:DnNDN}, respectively. Under the alternative hypothesis, if $I^{(N)}(r_0/N)$ is not constant in $r_0$, we have
	\[
	\frac{D_n^{(N)}}{\sqrt{k}}
	\dto
	\max_{r_0 = \Delta, \ldots, N-\Delta} I^{(N)}(r_0/N)
	-
	\min_{r_0 = \Delta, \ldots, N-\Delta} I^{(N)}(r_0/N)
	> 0,
	\]
	and thus $D_n^{(N)} \dto \infty$ as $n \to \infty$.
	
	\subsection{The limit distribution of the test statistic ($N \to \infty$)}
	
	Suppose that $N = N_n \to \infty$ and $\Delta = \Delta_n \to \infty$ in such a way that $\delta_n = \Delta_n/N_n \to \delta \in (0, 1/2)$. For simplicity, assume that $\delta_n = \delta$ for all sufficiently large $n$. To handle the more general case where $\delta_n \ne \delta$ but $\delta_n \to \delta$, proceed by an extension of the argument below, considering the asymptotic distributions of $\hat{I}$ below jointly in $\delta'$ in a neighborhoud around $\delta$. Introduce
	\[
	\hat{I}(t_0) = \frac{1}{2\delta}
	\int_{t_0-\delta}^{t_0+\delta} \hat{R}_{n;t, t_0}(1, 1) \, \diff t
	\]
	together with
	\[
	D_{n} = \sqrt{k} \left[
	\sup_{t \in [\delta,1-\delta]} \hat{I}(t_0)
	-
	\inf_{t \in [\delta,1-\delta]} \hat{I}(t_0)
	\right].
	\]
	Also write
	\[
	I(t_0) = \frac{1}{2\delta}
	\int_{t_0-\delta}^{t_0+\delta} R_{t, t_0}(1, 1) \, \diff t.
	\]
	Since
	\begin{equation}
	\label{eq:Vn}
	V_n(t_0)
	:=
	\sqrt{k} \{ \hat{I}(t_0) - I(t_0) \}
	=
	\frac{1}{2\delta} \int_{t_0-\delta}^{t_0+\delta} \sqrt{k} \{ \hat{R}_{n;t,t_0}(1, 1) - R_{t,t_0}(1, 1) \} \, \diff t,
	\end{equation}
	the continuous mapping theorem implies that, in the space $\ell^\infty([\delta,1-\delta])$, we have weak convergence
	\[
	V_n = \sqrt{k} (\hat{I} - I) \dto V, \qquad n \to \infty
	\]
	where $V$ is defined by
	\begin{equation}
	\label{eq:V}
	V(t_0) = \frac{1}{2\delta}
	\int_{t_0-\delta}^{t_0+\delta} \hat{W}(t,t_0;1,1) \, \diff t,
	\qquad t_0 \in [\delta,1-\delta].
	\end{equation}
	
	Under the null hypothesis of stationarity, $I(t_0)$ does not depend on $t_0 \in [\delta,1-\delta]$, and thus
	\begin{align*}
	D_n
	&=
	\sup_{t \in [\delta,1-\delta]} V_n(t_0)
	-
	\inf_{t \in [\delta,1-\delta]} V_n(t_0) \\
	&\dto
	\sup_{t \in [\delta,1-\delta]} V(t_0)
	-
	\inf_{t \in [\delta,1-\delta]} V(t_0)
	= D,
	\qquad n \to \infty.
	\end{align*}
	
	To show weak convergence $D_n^{(N)} \dto D$ as $n \to \infty$ as well, it is then sufficient to show that $D_n^{(N)} = D_n + o_p(1)$ as $n \to \infty$. By the assumption that $R_{s,t}(1, 1)$ is continuous in $(s, t)$, the empirical tail copula process is asymptotically uniformly equicontinuous in probability; this follows from Theorem~\ref{thm:main:2} and Lemma~\ref{lem:dotR:unifcont}.  The same then holds for the process $V_n$, whence
	\[
	D_n =
	\max_{r_0 = \Delta, \ldots, N-\Delta} V_n(r_0/N)
	-
	\min_{r_0 = \Delta, \ldots, N-\Delta} V_n(r_0/N)
	+
	o_p(1),
	\qquad n \to \infty.
	\]
	Further, the sum over $r$ in $V_n^{(N)}$ in \eqref{eq:VnN} can be written as a Riemann approximation to the integral over $t$ in $V_n$ in Eq.~\eqref{eq:Vn} by means of intervals of length $1/N$. Again, by asymptotic uniform equicontinuity in probability of the empirical tail copula process, we find that
	\[
	\max_{r_0 = \Delta, \ldots, N-\Delta}
	\left| V_n^{(N)}(r_0/N) - V_n(r_0/N) \right|
	= o_p(1), \qquad n \to \infty.
	\]
	The relation $D_n^{(N)} = D_n + o_p(1)$ as $n \to \infty$ follows.

	\subsection{Calculating the covariance matrix of $V^{(N)}$ for the Smith model and the Pareto process}
	\label{subsec:VN}
	
	The Gaussian random vector $V^{(N)}$ in Eq.~\eqref{eq:VN} is a linear combination of the random variables $\hat{W}(r/N, r_0/N; 1, 1)$ for $r_0 \in \{\Delta, N-\Delta\}$ and $r \in \{0, 1, \ldots, N\}$ such that $1 \le |r-r_0| \le \Delta$. The covariance matrix of $V^{(N)}$ can thus be calculated from the one of $\hat{W}(s, t; 1, 1)$ for $s, t \in \{0, 1/N, \dots, N\}$. In turn, the Gaussian process $\hat{W}$ in Theorem~\ref{thm:main:2} is the result of a linear transformation applied to the Gaussian process $W$ with covariance function in Eq.~\eqref{eq:cov}. For brevity of notation, we will temporarily omit the arguments $x_j = 1$ and write $W(s, t; 1, 1) \equiv W(s, t)$ and $\dot{R}_{s,t;j}(1, 1) \equiv \dot{R}_{s,t;j}$ and so on. It follows that
	\begin{align*}
	\lefteqn{
		\expec\left[ \hat{W}(s,t) \, \hat{W}(s',t') \right]
	} \\
	&=
	\expec[W(s, t) W(s', t')]
	-
	\dot{R}_{s',t'; 1} \expec[W(s, t) W(s')]
	-
	\dot{R}_{s',t';2} \expec[W(s, t) W(t')] \\
	&\qquad \mbox{}
	-
	\dot{R}_{s,t;1} \expec[W(s) W(s', t')]
	+
	\dot{R}_{s,t;1} \dot{R}_{s',t';1} \expec[W(s) W(s')]
	+
	\dot{R}_{s,t;1} \dot{R}_{s',t';2} \expec[W(s) W(t')] \\
	&\qquad \mbox{}
	-
	\dot{R}_{s,t;2} \expec[W(t) W(s', t')]
	+
	\dot{R}_{s,t;2} \dot{R}_{s',t';1} \expec[W(t) W(s')]
	+
	\dot{R}_{s,t;2} \dot{R}_{s',t';2} \expec[W(t) W(t')] \\
	&=
	R_{s,t,s',t'} - \dot{R}_{s',t';1} R_{s,t,s'} - \dot{R}_{s',t';2} R_{s,t,t'} \\
	&\qquad \mbox{}
	- \dot{R}_{s,t;1} R_{s,s',t'}
	+ \dot{R}_{s,t;1} \dot{R}_{s',t';1} R_{s,s'}
	+ \dot{R}_{s,t;1} \dot{R}_{s',t';2} R_{s,t'} \\
	&\qquad \mbox{}
	- \dot{R}_{s,t;2} R_{t,s',t'}
	+ \dot{R}_{s,t;2} \dot{R}_{s',t';1} R_{t,s'}
	+ \dot{R}_{s,t;2} \dot{R}_{s',t';2} R_{t,t'}.
	\end{align*}
	
	We calculate the quantities $\dot{R}_{s,t;j}(1, 1)$ and $R_{\vc{t}}(\vc{1})$ for vectors $\vc{t} \in [0, 1]^d$ with $d$ distinct elements for the Smith model in Example~\ref{exs} with $r = 1$ and $\Sigma = 1$ and for the Pareto process in Example~\ref{exp}. In both cases, the bivariate tail copulas are of H\"usler--Reiss form, so that
	\[
	\dot{R}_{s,t;j}(1, 1) =
	\begin{cases}
	\bar{\Phi}(|s-t|/2), & \text{Smith model,} \\
	\bar{\Phi}(\sqrt{|s-t|}/2), & \text{Pareto process.}
	\end{cases}
	\]
	
	Next we calculate the tail dependence coefficients $R_{\vc{t}}(\vc{1})$. For the Smith model, we have, for a standard normal variable $Z$,
	\[
	R_{\vc{t}}(\vc{1})
	=
	\expec\left[
	\min_{j=1,\ldots,d}
	\exp\left(t_jZ - t_j^2/2\right)
	\right]
	=
	2 \bar{\Phi}\left(\tfrac{1}{2}(t_{(d)} - t_{(1)})\right)
	\]
	where $t_{(1)}$ and $t_{(d)}$ are the minimum and the maximum of $\vc{t}$, respectively. To show the second equality above, split the expectation according to the index $j$ where the minimum is attained and use the identity $\expec[\exp(uZ - u^2/2)f(Z)] = \expec[f(Z+u)]$ for scalar $u$ and measurable nonnegative function $f$.
	
	For the Pareto process, we have
	\[
	R_{\vc{t}}(\vc{1})
	=
	\expec\left[
	\min_{j=1,\ldots,d}
	\exp\left\{
	W'(t_j) - t_j^2/2
	\right\}
	\right]
	\]
	where $W'$ is a standard Wiener process. The expectation can be calculated by means of the following lemma, using that $\cov\{W'(s), W'(t)\} = \min(s, t)$ for $s, t \in [0, 1]$. If some $t_j$ equals $0$, then the covariance matrix of $(W'(t_j))_j$ is not positive definite, but the formula still holds, as can be seen for instance by replacing $\vc{t}$ by $\vc{t} + c$ for some $c > 0$ and exploiting the fact that $W'$ has stationary and independent increments.
	
	\begin{lemma}
		\label{lem:Eminexp}
		For a $d$-dimensional ($d \ge 2$) centered Gaussian random vector $X$ with positive definite covariance matrix $\Gamma = (\gamma_{jk})_{j,k=1}^d$, we have
		\begin{align*}
		\expec\left[
		\min_{j=1,\ldots,d}
		\exp\left(X_j - \tfrac{1}{2}\gamma_{jj}\right)
		\right]
		=
		\sum_{j=1}^d \Phi_{d-1}\left(
		\bigl(-\tfrac{1}{2}\var(\Delta_k^{(j)})\bigr)_{k:k\ne j};
		\cov(\Delta^{(j)})
		\right)
		\end{align*}
		where $\Phi_{r}(\,\cdot\,;\Sigma)$ is the cdf of the $r$-variate centered normal distribution with covariance matrix $\Sigma$ and where $\Delta^{(j)} = (\Delta^{(j)}_k)_{k:k\ne j}$ with $\Delta^{(j)}_k = X_j - X_k$.
	\end{lemma}
	
	\begin{proof}
		For every vector $a \in \reals^d$ and for every nonnegative measurable function $f$, we have $\expec[\exp(a^\top X - \tfrac{1}{2} a' \Gamma a) f(X)] = \expec[f(X + \Gamma a)]$. It follows that
		\begin{align*}
		\lefteqn{
			\expec\left[
			\min_{j=1,\ldots,d}
			\exp\left(X_j - \tfrac{1}{2} \gamma_{jj}\right)
			\right]
		} \\
		&= \sum_{j=1}^d \expec\left[
		\exp\left(X_j - \tfrac{1}{2} \gamma_{jj}\right) \,
		\mathds{1} \left\{
		X_j - \tfrac{1}{2} \gamma_{jj}
		\le \min_{k:k\ne j} (X_k - \tfrac{1}{2} \gamma_{kk})
		\right\}
		\right] \\
		&=
		\sum_{j=1}^d \prob \left[
		X_j + \tfrac{1}{2}\gamma_{jj}
		\le \min_{k:k\ne j} \{X_k + \gamma_{kj} - \tfrac{1}{2}\gamma_{kk}\}
		\right]
		\end{align*}
		which can be rearranged into the stated expression.
	\end{proof}
	
	Note that the expression in Lemma~\ref{lem:Eminexp} is akin to the one of the stable tail dependence function or exponent function of the H\"usler--Reiss copula in Remark~2.5 in \citet{nikoloulopoulos+j+l:2009}.

	\subsection{Nonparametric estimation the covariance matrix of $V^{(N)}$}
	\label{subsec:VNnonparam}
	
	To estimate the covariance matrix of $V^{(N)}$ in Eq.~\eqref{eq:VN} we start as Section~\ref{subsec:VN} above and express $V^{(N)}$ as a linear transformation of the centered Gaussian process $\hat{W}$, which is in turn a linear transformation of the centered Gaussian process $W$. The covariance function of $W$ in Eq.~\eqref{eq:cov} is given entirely in terms of the tail copula, which we estimate by the empirical tail copula.
	
	The coefficients appearing in the linear transformation are all known except for the partial derivatives $\dot{R}_{s,t;j}(1, 1)$. These can be estimated by a finite differencing via
	\[
	\frac{1}{2\eta} \left\{
	\hat{R}_{n;s,t}(1+\eta, 1) - \hat{R}_{n;s,t}(1-\eta, 1)
	\right\}
	\]
	for $j = 1$ and similarly for $j = 2$; here $\eta = \eta_n > 0$ is such that $\eta > 0$ but $k^{1/2} \eta \to \infty$ as $n \to \infty$. Under the null hypothesis of stationarity, the partial derivative only depends on $(s, t)$ via $s-t$. This information can exploited by averaging the estimator over all pairs $(s, t)$ with the same value of $s - t$. Also, since $0 \le \dot{R} \le 1$, we can enforce the same constraints on the estimate.
	
	The estimate of the covariance matrix of a finite-dimensional distribution of $W$ based on the empirical tail copula is automatically symmetric and positive semi-definite. This can be seen by writing the empirical tail copula as a matrix of second-order moments of indicator variables. The symmetric and positive semi-definite character is then inherited by the estimate of the covariance matrix of $V^{(N)}$ by using the formula that $\cov(A X) = A \cov(X) A'$ for a random vector $X$ and a real matrix $A$ with the appropriate number of columns.
	
	Under the null hypothesis of stationarity, the covariance matrix of $V^{(N)}$ is Toeplitz, which is in general not the case for the estimate, although we found in the simulations that it was usually not far from being so, perhaps because of the regularisation of the estimate of $\dot{R}$ as explained above. To enforce the estimated covariance matrix to be Toeplitz, we averaged out over the subdiagonals. This operation never destroyed the positive definite property, although it is not guaranteed to do so. If needed, positive definiteness can be restored by adding a diagonal matrix with a small but sufficiently large diagonal depending on the smallest eigenvalue of the estimated matrix.

	\subsection{Calculating $p$-values}
	The $p$-value associated to the test statistic statistic $D_n^{(N)}$ in Eq.~\eqref{eq:DnN} and computed from an estimate $\hat{\Sigma}_n^{(N)}$ of the covariance matrix of $V^{(N)}$ as described in Section~\ref{subsec:VNnonparam} is given by
	\[
	1 - F\bigl(D_n^{(N)}; \hat{\Sigma}_n^{(N)}\bigr),
	\]
	where for $q \in [0, \infty)$ and for a covariance matrix $\Gamma$ of dimension $m \times m$ with $m \ge 2$, we put
	\[
	F(q; \Gamma) = \prob[\max(U) - \min(U) \le q]
	\]
	with $U$ a centered multivariate normal random vector with covariance matrix $\Gamma$. The latter probability can be calculated as follows, provided $\Gamma$ is positive definite: partitioning the event according to the index $i \in \{1,\ldots,m\}$ at which the maximum is realized, we have
	\begin{align}
	\label{eq:range(U)}
	\prob[ \max(U) - \min(U) \le q ]
	&= \sum_{i=1}^m \prob\left[
	\forall j \in \{1,\ldots,m\} \setminus \{i\} :
	-q < U_j - U_i < 0
	\right] \\
	\nonumber
	&= \sum_{i=1}^m \Phi_{m-1}\left(
	[-q, 0]^{m-1} ;
	\cov\left((U_j - U_i)_{j:j \ne i}\right)
	\right),
	\end{align}
	in terms of the $r$-variate centered Gaussian probability measure $\Phi_r$ with the given covariance matrix. Substituting $q = D_n^{(N)}$ and $\Gamma = \hat{\Sigma}_n^{(N)}$ then yields the $p$-value appearing in the middle row in Figure~\ref{fig:testStat} in the paper. It is also on the basis of these $p$-values that the hypothesis test is implemented of which the power is shown on the bottom row of the same figure.
	
	In the numerical experiments, it was the computation of the $p$-value that was the most time-consuming, especially if $N$ is large: for $N = 20$ and $\Delta = 2$, the dimension of $V^{(N)}$ is $N - 2\Delta + 1 = 17 = m$, so that computing a single $p$-value requires $m = 17$ evaluations of the multivariate Gaussian probability measure of dimension $m-1 = 16$. We used the function \textsf{sadmvn} of the R-package \textsf{mnormt} to perform this task, which is in turn based on an algorithm described in \citet{genz:1992}. The calculation of a single $p$-value then took about $1$ second on a standard laptop, which is not that much, but of course quickly adds up if to be repeated over many samples.

	\subsection{Calculating the probability density function of the limit $D^{(N)}$}
	
	The probability density functions at the top row of Figure~\ref{fig:testStat} are, upon rescaling by $2\Delta\sqrt{k}$, the ones of the limit variable $D^{(N)} = \max(V^{(N)}) - \min(V^{(N)})$ for the Smith model and the Pareto process. The cumulative distribution function of $D^{(N)}$ is given in Eq.~\eqref{eq:range(U)}, where for $\Gamma$ we substitute the covariance matrix of $V^{(N)}$ for the given model, as calculated  in Subsection~\ref{subsec:VN}. From the cdf we then proceed to calculate the pdf by numeric differentiation (finite differencing). An alternative, which we did not try, would be to analytically differentiate Eq.~\eqref{eq:range(U)} with respect to $q$ and calculate the resulting expression numerically.

\end{appendix}

\section*{Acknowledgements}
We are very grateful to an Associate Editor and three Referees for various insightful comments that led to this improved version of the manuscript.

John Einmahl holds the Arie Kapteyn Chair 2019--2022 and gratefully acknowledges the corresponding research support.

\bibliographystyle{imsart-nameyear}
\bibliography{biblio}
\end{document}